\documentclass[9 pt,reqno]{article}
\usepackage{amssymb, amsmath}
\usepackage{amssymb,amsthm}
\newtheorem{Theorem}{Theorem}
\newtheorem{Remark}{Remark}
\newtheorem{Lemma}{Lemma}

\newtheorem{Corollary}{Corollary}
\newtheorem{Definition}{Definition}
\numberwithin{Definition}{section}
\numberwithin{Theorem}{section}
\numberwithin{Lemma}{section}
\numberwithin{Proposition}{section}

\numberwithin{Corollary}{section}
\numberwithin{Example}{section}
\numberwithin{Remark}{section}
\begin{document}
\title{ Enhanced Fritz John Stationarity, New Constraint Qualifications and Local Error Bound for Mathematical Programs with Vanishing Constraints}
	
	\author{Abeka Khare\thanks{Department of Mathematics and Statistics, Dr. Harisingh Gour Vishwavidyalaya, Sagar, Madhya Pradesh-470003, INDIA, Email- abekakhare2012@gmail.com}  and  Triloki Nath\thanks{Department of Mathematics and Statistics, Dr. Harisingh Gour Vishwavidyalaya, Sagar, Madhya Pradesh-470003, INDIA, Email- tnverma07@gmail.com} }
	\maketitle
\begin{abstract}
In this paper, we study the difficult class of optimization problems called the mathematical programs with vanishing constraints or MPVC. Extensive research has been done for MPVC regarding stationary conditions and constraint qualifications using geometric approaches. We use the Fritz John  approach for MPVC to derive the M-stationary conditions under weak  constraint qualifications. An enhanced Fritz John type stationary condition is also derived for MPVC, which provides the notion of enhanced M-stationarity under a new and weaker constraint qualification: MPVC-generalized quasinormality. We show that this new constraint qualification is even weaker than MPVC-CPLD. A local error bound result is also established under MPVC-generalized quasinormality.
\end{abstract}


\section{Introduction}
\label{sec:intro}
In this paper, we consider a particular form of optimization problem which attracted the attention of the optimization community over the past decade. It has the following  form
\begin{eqnarray}
\label{eq1.2}
\nonumber {\rm min} ~~ f(x)\\
\nonumber {\rm s.t.}~~ g_i(x) &\leqslant&  0 ~~~\forall~  i=1,2,...,m,\\
\nonumber h_j(x) &=&  0 ~~~\forall~  j=1,2,...,p,\\
\nonumber	H_i(x) &\geqslant&  0 ~~~\forall~  i=1,2,...,q,\\
		G_i(x) H_i(x) &\leqslant&  0 ~~~\forall~  i=1,2,...,q
	\end{eqnarray}
where all functions $f: \mathbb{R}^n \rightarrow \mathbb{R}, ~g_i: \mathbb{R}^n \rightarrow \mathbb{R},~h_i: \mathbb{R}^n \rightarrow \mathbb{R},~G_i: \mathbb{R}^n \rightarrow \mathbb{R},~H_i:\mathbb{R}^n \rightarrow \mathbb{R}$ are continuously differentiable.
 This problem (\ref{eq1.2}) is called \emph{mathematical program with vanishing constraints} (or MPVC) due to its implicit sign constraints $G_i(x) \leqslant 0$, which vanishes immediately whenever $H_i(x) = 0$ (Indeed, $H_i(x) = 0 \Longrightarrow G_i(x) \in \mathbb R$, \emph{i.e.} $G_i(x) H_i(x) \leqslant  0$ is no more a constraint). We denote by $\mathcal{C}$,  the feasible region for this MPVC throughout the paper.\\
 \hspace*{0.5cm} The above formulation of the problem has been first introduced and studied by Achtziger and Kanzow in \cite{achtziger}, where the structure of the problem is discussed in a very lucid and systematic way. Before \cite{achtziger}, a few  papers in engineering applications had appeared e.g. \cite{achtziger 1,bendsoe,cheng,kirsch}, which considered particular cases of this general problem, but \cite{achtziger} was the first formal treatment of MPVC, it also featured some applications of MPVC like ground structure and truss structure design. The applications are not limited only to structural and topology optimization, but also applicable in robots motion planning \cite{Kirches2,Latombe}. These applications of MPVC motivate the researchers for further study  to  evolve the field  and to find some new tools to tackle this problem from theoretical as well as algorithmic point of view, e.g. \cite{achtziger2,jabr,michael}. Subsequent to \cite{achtziger}, a lot of collaborative work has been done by Hoheisel and Kanzow in \cite{hoheisel07,hoheisel 2,hoheisel 4}, which are the detailed study about the constraint qualifications and optimality conditions for MPVC. For more literature on MPVC, we refer to \cite{hoheisel07,hoheisel,Izmailov} and references therein. \\\\
 \hspace*{0.5cm} Another well-studied class of optimization problems is \emph{mathematical programs with equilibrium constraints} (MPEC)
 \begin{eqnarray}
\label{eq1.3}
\nonumber {\rm min} ~~ f(x)\\
\nonumber {\rm s.t.}~~ g_i(x) &\leqslant&  0 ~~~\forall~  i=1,2,...,m,\\
\nonumber h_j(x) &=&  0 ~~~\forall~  j=1,2,...,p,\\
\nonumber	G_i(x) &\geqslant&  0 ~~~\forall~  i=1,2,...,q,\\
\nonumber	H_i(x) &\geqslant&  0 ~~~\forall~  i=1,2,...,q,\\
		 G_i(x) H_i(x) &=&  0 ~~~\forall~  i=1,2,...,q
\end{eqnarray}
where all functions $f : \mathbb{R}^n \rightarrow \mathbb{R}, ~g_i : \mathbb{R}^n \rightarrow \mathbb{R},~h_i : \mathbb{R}^n \rightarrow \mathbb{R},~G_i : \mathbb{R}^n \rightarrow \mathbb{R},~H_i: \mathbb{R}^n \rightarrow \mathbb{R}$ are continuously differentiable. The MPEC is known to be a difficult optimization problem due to the violation of the standard constraint qualifications (CQs), e.g. the linear independence constraint qualification (LICQ) and the Mangasarian-Fromovitz constraint qualification (MFCQ) at any feasible point. It possibly happens due to disjunctive and combinatorial nature of characteristic constraints (e.g. $H_i(x)\geqslant 0, G_i(x) H_i(x)\leqslant 0$ in MPVC). Hence, the classical KKT conditions of standard nonlinear programming are not always necessary
optimality conditions for MPEC, even in the case when all constraint functions are affine.
In order to find the first-order optimality conditions for an MPEC, modified constraint qualifications called MPEC-tailored constraint qualifications are defined. We refer to \cite{Luo1,Pang,Scheel,ye4} for a comprehensive overview. One of the optimality conditions called strong stationarity (S-\emph{stationarity}) is equivalent to the KKT conditions of an MPEC \cite{flegel05} (c.f. \cite{achtziger} for MPVC). Hence, the S-stationarity is not always a necessary optimality condition. A slightly weaker  notion called M-\emph{stationarity} (see \cite{Outrata1,Outrata2,ye 2,ye5}) is
first-order necessary optimality conditions, which hold under mild assumptions, see \cite{flegel06,flegel07,ye 3}).
To overcome this  difficulty, the  violation of standard constraint qualifications, the Fritz John approach is useful. Because it is classical that  Fritz John necessary conditions do not require any constraint qualification. This approach
is used in \cite{flegel}, to provide a simple proof for A-\emph{stationarity} ( weaker than M-stationarity) to be a necessary optimality condition under MPEC-MFCQ. In \cite{ye 3} (see also \cite{flegel06}),  based on the limiting subdifferential and the limiting coderivatives by Mordukhovich, see \cite{Rocka,mordu,mordu1} to grasp such concepts, the Fritz John approach is used to find M-stationarity which is most appropriate necessary optimality condition for MPEC.\\
 \hspace*{0.5cm} To this end, we note that it has been pointed out in \cite {achtziger} that MPVC can always be reformulated as MPEC. But it has some drawbacks, particularly since it increases the dimension of the problem, and it involves locally non-unique solutions of corresponding MPEC, also it loses its characteristic of vanishing constraints of MPVC. So, it suggests to investigating MPVC independently taking into account the special structure of vanishing constraints. On the other hand, it also suggests that the whole MPEC machinery or analogous to that can be applied  to an MPVC. Thus, Fritz John approach too can be used to exploit the special structure of vanishing constraints of an MPVC. Using the approach for MPVC, we show that M-stationary conditions hold under MPVC-MFCQ, MPVC-\emph{linear}-CQ and a weaker constraint qualification MPVC-GMFCQ.\\

The \emph{enhanced} Fritz John conditions were first introduced by Bertsekas \cite{bert} (a weaker version of this is given by Hestenes \cite {hestenes}) is stronger than the classical Fritz John conditions. Following \cite{bert}, Kanzow and Schwartz \cite{Kanzow1}
established the enhanced KKT type  stationarity  (called enhanced M-stationarity ) conditions for MPEC under weaker MPEC-constraint qualifications, namely MPEC \emph{generalized pseudonormality} and \emph{quasinormality}. These constraint qualifications are MPEC version of that introduced in \cite{bertsekas2}  for standard nonlinear programs. In \cite{ye}, the results of  \cite{Kanzow1} have been extended to the nonsmooth case. In \cite{Kanzow1}, it has been shown that pseudonormality is a sufficient condition for the existence of \emph{local error bound} for MPEC. Whereas in \cite{ye}, it has been improved to the nonsmooth case by showing that the MPEC-generalized quasinormality is sufficient condition for the same under fairly mild assumptions because pseudonormality implies quasinormality. The MPVC-\emph{generalized pseudonormality} has been introduced in \cite{hu}, we introduce a new constraint qualification: MPVC-\emph{generalized quasinormality}, its MPEC variant is known. We prove that an enhanced M-stationary condition holds under this constraint qualification.\\

     In recent years, It has been shown that  many constraint qualifications such as \emph{pseudonormality} and \emph{quasinormality} \cite{ye 1,bertsekas2}, \emph{constant
positive linear dependence} (CPLD), see  \cite{qi} and \emph{relaxed constant positive linear
dependence} (RCPLD), see \cite{andreani},  all  to be weaker than MFCQ. Following MPVC-CPLD defined in \cite{hohei con}, we show that it is stronger than the MPVC-generalized quasinormality.
Further, we show that the MPVC-CPLD is also a constraint
qualification for the enhanced M-stationarity and provides a sufficient condition for
the existence of a local error bound for the MPVC.\\\\

\hspace*{0.5cm} We organize the present paper as follows. In section 2, we recall some well-known MPVC tailored constraint qualifications. We also recall some well-known stationarity notions and few definitions from nonsmooth analysis. Section 3 is devoted to Fritz John type stationary conditions, which lead to KKT type stationary conditions under suitable constraint qualifications. In section 4, enhanced Fritz John type stationary conditions are investigated in depth. As a result, KKT type enhanced stationary conditions are discussed under the various known and a new constraint qualification: MPVC-generalized quasinormality. Further, it is shown that MPVC-CPLD is stronger than MPVC-generalized quasinormality. The section 4 discusses the error bound results and finally we provide concluding remarks in section 5.

\section{Preliminaries}\label{sec:Prelim}
Here, we provide some relevant definitions and background material for the MPVC formulated in (\ref{eq1.2}), that will be used in subsequent sections of this paper. For an arbitrary feasible point $x^\ast$  we adopt the following notations for index sets from \cite{hoheisel 2,hu}, which are analogous to MPEC \cite{Kanzow1,ye}. We define the index sets as follows
\begin{eqnarray*}
		I_g(x^\ast) := \{i ~|~ g_i(x^\ast) = 0\}, ~~ I_{+}(x^\ast) := \{i~|~H_i(x^\ast) > 0\}, ~{\rm and~} I_{0}(x^\ast) := \{i~|~H_i(x^\ast) = 0\}	.	
	\end{eqnarray*}
The index set $I_{+}(x^\ast)$ is further divided into two subsets
\begin{eqnarray*}
		I_{+0}(x^\ast) &:=& \{i~|~H_i(x^\ast) > 0 ~,~ G_i(x^\ast) = 0\},\\
		I_{+-}(x^\ast) &:=& \{i~|~H_i(x^\ast) > 0 ~,~ G_i(x^\ast) < 0\}
\end{eqnarray*}
and the set $I_{0}(x^\ast)$ can be partitioned as follows
\begin{eqnarray*}
	    I_{00}(x^\ast) &:=& \{i~|~H_i(x^\ast) = 0 ~,~ G_i(x^\ast) = 0\},\\
		I_{0+}(x^\ast) &:=& \{i~|~H_i(x^\ast) = 0 ~,~ G_i(x^\ast) > 0\},\\
		I_{0-}(x^\ast) &:=& \{i~|~H_i(x^\ast) = 0 ~,~ G_i(x^\ast) < 0\}.
	\end{eqnarray*}
If the concerned point is understood, we denote the index sets simply by $I_g,~ I_{+}, ~I_{+0},~I_{00}$ and so on.\\
\hspace*{0.5cm}Now, we recall  some standard constraint qualifications for MPVC based on these notations. The following two constraint qualifications were formally introduced in \cite{hoheisel 2}.
	\begin{Definition}
	\label{Def1} A vector $x^\ast \in \mathcal{C}$ is said to satisfy MPVC-linearly independent constraint qualification  (or MPVC-LICQ) if the gradients \\
	$~~~~~~~~~~~\{\nabla g_i(x^\ast) | i \in I_g(x^\ast)\} ~\cup ~ \{\nabla h_i(x^\ast) | i = 1,...,p\} ~\cup ~ \{\nabla G_i(x^\ast) | i \in I_{+0}(x^\ast)~\cup~ I_{00}(x^\ast)\} \\~~~~~~~~~~~ \cup~ \{\nabla H_i(x^\ast) | i \in I_0(x^\ast)\} $ \\
	are linearly independent.
	\end{Definition}
	 \begin{Definition}
\label{Def2}
	 	  A vector $x^\ast \in \mathcal{C}$ is said to satisfy MPVC-Mangasarian Fromovitz constraint qualification (or MPVC-MFCQ) if the \\
	 	$~~~~~~~~~~~~~~\nabla h_i(x^\ast) ~~(i= 1, ..., p),~~~~\nabla H_i(x^\ast)~~(i \in I_{0+}(x^\ast) \cup I_{00}(x^\ast))$\\
	 	are linearly independent and there exists a vector $d \in \mathbb{R}^n$ such that \\
	 	$~~~~~~~~\nabla h_i(x^\ast)d =0 ~~~(i=1,...,p), ~~~~\nabla H_i (x^\ast)^T d = 0 ~~~(i \in I_{0+}(x^\ast) \cup I_{00}(x^\ast) ),\\
	 	~~~~~~~\nabla g_i (x^\ast)^T d < 0 ~~~(i \in I_g(x^\ast)), ~~~~\nabla H_i (x^\ast)^T d > 0 ~~~(i \in I_{0-}(x^\ast)),\\
	 ~~~~~~~	\nabla G_i(x^\ast)^T d < 0 ~~~(i \in I_{+0}(x^\ast) \cup I_{00}(x^\ast))$.
	 \end{Definition}
It has been seen that these constraint qualifications are very useful. In \cite [Theorem 4.5]{hoheisel}, MPVC-MFCQ is shown to be a sufficient condition for exactness of MPVC-exact penalty function.\\
\hspace*{0.5 cm} In the spirit of MPEC-GMFCQ  \cite{ye 3}, the following MPVC-GMFCQ was introduced in \cite{hu}, and has been shown to be a key in exact penalty results.
	\begin{Definition}
	\label{Def3} A vector $x^\ast \in \mathcal{C}$ is said to satisfy MPVC-generalized MFCQ (or MPVC-GMFCQ) if there is no multiplier $(\lambda, \mu, \eta^H, \eta^G)~\neq~0$ such that \\
	
	$ {\rm (i)}~\sum_{i=1}^{m} \lambda_i \nabla g_i(x^\ast) + \sum_{i=1}^{p} \mu_i \nabla h_i (x^\ast) + \sum_{i=1}^{q} \eta^G_i \nabla G_i(x^\ast)
	-\sum_{i=1}^{q} \eta^H_i \nabla H_i(x^\ast)~=~0,$\\

$ {\rm (ii)} ~~~~\lambda_i \geqslant 0 ~~\forall~i \in I_g(x^\ast),~~~\lambda_i = 0 ~~~\forall~i \notin I_g(x^\ast)$
\begin{eqnarray*}
{\rm and}~~~\eta_i^G &=& 0 ~~\forall~i \in I_{+-}(x^\ast) \cup I_{0-}(x^\ast) \cup I_{0+}(x^\ast),~~\eta_i^G  \geqslant 0~\forall i \in I_{+0}(x^\ast) \cup I_{00}(x^\ast),\\
\eta_i^H &=& 0 ~~\forall~i \in I_+(x^\ast),~~\eta_i^H  \geqslant  0~\forall ~i \in I_{0-}(x^\ast)~~ {\rm and}~~\eta_i^H ~ {\rm is}~ {\rm free}~ \forall~ i \in I_{0+}(x^\ast),\\
\eta_i^H  \eta_i^G &=& 0~~\forall~ i \in I_{00}(x^\ast).
\end{eqnarray*}
\end{Definition}
\begin{Remark}	In {\rm \cite[Proposition 2.1]{ye 3}}, it has been established that MPEC-generalized MFCQ is equivalent to NNAMCQ. Analogously, we can show it for MPVC, if NNAMCQ is defined for MPVC analogous to MPEC notion. So we can identify  MPVC-GMFCQ and MPVC-NNAMCQ to be the same.
\end{Remark}
We note from \cite[Proposition 2.1]{hu} that following implications hold:\\\\
	\hspace*{5 cm}MPVC-LICQ $\Rightarrow$ MPVC-MFCQ $\Rightarrow $ MPVC-GMFCQ.\\\\
\hspace*{0.5 cm}The following stationarity concepts are widely studied in the literature \cite{achtziger,hoheisel 2,hoheisel} and known to be important optimality conditions for the MPVC. These stationary conditions differ only for the multipliers associated with the indices in $I_{00}$, see \cite{Dussault}.
 	\begin{Definition}
 	\textbf{W-Stationary Condition:} Any feasible point $x^\ast$ for $(P)$ is called \textit{W-Stationary point} for MPVC, if there is a multiplier $(\lambda, \mu, \eta^G, \eta^H)$ such that
 	\begin{eqnarray*}
 	0~=~\nabla f(x^\ast)~+~\sum_{i \in I_g}^{} \lambda_i \nabla g_i(x^\ast)~+~ \sum_{i \in I_h}^{} \mu_i \nabla h_i(x^\ast)~+~\sum_{i=1}^{q} \eta_i^G \nabla G_i(x^\ast)~-~\sum_{i=1}^{q} \eta_i^H \nabla H_i(x^\ast)
 	\end{eqnarray*}
\begin{eqnarray*}
 {and}~~~\lambda_i  &\geqslant & 0 ~~\forall~i \in I_g(x^\ast),~~
 		\lambda_i = 0~~\forall~i \notin I_g(x^\ast),\\	
	\eta_i^G &=& 0 ~~\forall~i \in I_{+-}(x^\ast) \cup I_{0-}(x^\ast) \cup I_{0+}(x^\ast),~~\eta_i^G  \geqslant 0~\forall i \in I_{+0}(x^\ast) \cup I_{00}(x^\ast),\\
\eta_i^H &=& 0 ~~\forall~i \in I_+(x^\ast),~~\eta_i^H  \geqslant  0~\forall ~i \in I_{0-}(x^\ast)~~ {\rm and}~~\eta_i^H ~ {\rm is}~ { \rm free}~ \forall~ i \in I_{0+}(x^\ast).\\
\end{eqnarray*}
{\rm Now $x^\ast$ becomes}
\begin{enumerate}
\item \textbf{M-Stationarity:} If at $x^\ast$ W-Stationarity holds with
\begin{eqnarray*}
\eta_i^H  \eta_i^G  =  0~~\forall~ i \in I_{00}(x^\ast).
\end{eqnarray*}
\item \textbf{S-Stationarity:} If at $x^\ast$ W-Stationarity holds and
\begin{equation*}
\eta_i^H ~\geqslant~0 , \eta_i^G =0~~\forall~ i \in I_{00}(x^\ast).
\end{equation*}
\end{enumerate}
\end{Definition}
``M-stationarity'' is the most important stationarity concept from a theoretical perspective as it holds under minimal assumption. There is another stationarity concept also, such as T-\emph{stationarity},  see \cite{Dussault}. The T-\emph{stationarity} is a counterpart of C-\emph{stationarity} for MPEC. It is easy to see that the W- and T- stationarity notions are weaker than M-stationarity.  They all differ in the sign of
$\eta_i^H  $ and $\eta_i^G$  for the indices $i \in I_{00}(x^\ast)$. In this paper,  the W- and T- stationarity will play no role, viewing them typically as too weak for our purposes.\\

  Here, it is worth mentioning that the only difference between the \textit{S-Stationarity} (which is equivalent to standard KKT condition) and \textit{M-Stationarity} is in the multiplier $\eta_i^G , ~i \in I_{00}$. In \textit{S-Stationarity} $\eta_i^G = 0, ~\forall ~i \in I_{00}$, whereas in \textit{M-Stationarity} $\eta_i^G \geq 0, ~\forall ~i \in I_{00}$. The positive $\eta_i^G , ~i \in I_{00}$ in \textit{M-Stationarity} may play some significant role in MPVC-analysis. Clearly, \textit{M-Stationarity} is slightly weaker than \textit{S-Stationarity}. Further, \textit{S-Stationarity} holds at least under MPVC-LICQ, while \textit{M-Stationarity} holds under some weaker assumptions called MPVC-tailored CQ's like MPVC-GCQ or MPVC-MFCQ. \\

  Now, we recall some basic tools from the nonsmooth analysis. We give only concise definitions and results that will be used later to prove our main results. We refer to Mordukhovich \cite{mordu}, Rockafellar and Wets \cite{Rocka} and Clarke \cite{clarke} for more detailed information on the subject.\\\\
\hspace*{0.5 cm} First, we mention that throughout the paper, the following notations will be used. The symbol $\langle. ,.\rangle$,  $||.||$ and $||.||_1$ denote the standard inner product, Euclidean and max norm on $\mathbb R^n$ respectively. For a function $g:\mathbb{R}^n \rightarrow \mathbb{R},~g^+(x):= \max \{0, g(x)\}$, here $g^+$ denotes a vector if $\max$ function is defined componentwise.\\\\
\hspace*{0.5cm} Here are some definitions of various cones, which are known to be important tools in variational analysis.

\begin{Definition}{\rm
\begin{enumerate}
\item Let $C \subset \mathbb{R}^n$ be a nonempty set. The  \emph{polar cone} of $C$ is defined as
	  \begin{equation*}
 	C^\circ := \{ s \in  \mathbb{R}^n ~|~ s^Td ~\leqslant~ 0 ~~ \forall~ d \in C \}.
	  \end{equation*}
	  \item Let $C \subset \mathbb{R}^n$ be a nonempty closed set and $x^\ast \in C$. The (Bouligand)  \emph{tangent cone} (or {\rm contingent cone}) of $C$ at $x^\ast$ is defined as
\begin{eqnarray*}
	  	T_C(x^\ast) := \{d \in \mathbb{R}^n ~|~ \exists \{x^k\} \rightarrow_C x^\ast, ~\{t_k\} \downarrow 0 ~:~\frac{x^k - x^\ast}{t_k} \rightarrow d \}
\end{eqnarray*}
	  where $\{x^k\} \rightarrow_C x^\ast$ denotes a sequence $\{x^k\}$ converging to $x^\ast$ and satisfying $x^k \in C ~\forall~ k \in \mathbb{N}$.
\end{enumerate}}
 \end{Definition}
 \begin{Definition} {\rm
 Let $C \subset \mathbb{R}^n$ be a nonempty closed set and $x^\ast \in C$.
 \begin{enumerate}
\item The \emph{Fr\'echet normal cone} of $C$ at $x^\ast$ is defined as
	  \begin{equation*}
	  N_C^F(x^\ast) := T_C(x^\ast)^\circ .
	  \end{equation*}
\item    The convex cone
\begin{equation*}
  N_C^{\pi}(x^\ast) :=  \{ d \in \mathbb R^n | \exists \sigma > 0  ~ s. t.~ \langle d,x - x^{\ast} \rangle \leqslant \sigma  { ||x - x^{\ast} ||}^2~ \forall ~ x \in C \}
\end{equation*}
  is called the  \emph{proximal normal cone} to $C$ at  $x^\ast$.
\item  The \emph{limiting normal cone} of $C$ at $x^\ast$ is defined as	
\begin{eqnarray}
\label{eq1.4}
\nonumber N_C(x^\ast)&:=&\{d \in \mathbb{R}^n ~|~\exists \{x^k\} \rightarrow_C x^\ast, d^k \in N^F_C(x^k)~:~ d^k \rightarrow d \}\\
\nonumber &=&  \{d \in \mathbb{R}^n ~|~\exists \{x^k\} \rightarrow_C x^\ast, d^k \in N^{\pi}_C(x^k)~:~ d^k \rightarrow d \}.
	\end{eqnarray}
	The last identity can be seen in \cite{Rocka}.
\end{enumerate}
	 } \end{Definition}	
\section{Fritz John type Stationary Conditions}
\label{fjsc}
	To obtain  first-order necessary optimality conditions for a standard nonlinear program, usually three different approaches are available in the literature:  (i) geometrical approach (tangent cone criterion), (ii) exact penalization approach \cite{clarke}  and (iii) Fritz John conditions. The third approach is preferred, sometimes, over others, since it does not require any constraint qualification. But, it has a  disadvantage, it involves Lagrange multiplier associated to the objective function, which may be zero. However, this disadvantage, innovatively suggests new constraint qualifications, so that KKT conditions are satisfied, see \cite{manga,bert,ye 1}. However, they lead to same KKT conditions under suitable constraint qualifications. It has been verified in MPEC case also for important stationarity conditions,  see \cite {flegel,ye 3,Kanzow1}. So, we have enough motivation to approach the MPVC via Fritz John  conditions. In \cite[Theorem 2.1]{ye 3}, Ye has derived a Fritz John type M-stationary condition for MPEC.  Analogously, we do this for MPVC by reformulating it into an equivalent form EMPVC.
\begin{Theorem}\textbf{A Fritz John type M-stationary condition:} Let $x^\ast$ be a local minimum of MPVC, then there exist  $\alpha \geqslant 0~ {\rm and} ~\lambda, ~\mu, ~\eta^H, ~\eta^G$, not all zero, such that \\
 $ {\rm (i)} ~~\alpha \nabla f(x^\ast) ~+~ \sum_{i=1}^{m} \lambda_i \nabla g_i(x^\ast) ~+~ \sum_{i=1}^{p} \mu_i \nabla h_i (x^\ast) ~+~ \sum_{i=1}^{q} \eta^G_i \nabla G_i(x^\ast)~-~ \sum_{i=1}^{q} \eta^H_i \nabla H_i(x^\ast)~=~0 $;\\\\
$ {\rm (ii)} ~~~~~\lambda_i \geqslant 0 ~~~\forall~i \in I_g(x^\ast),~~~\lambda_i = 0 ~~~\forall~i \notin I_g(x^\ast)$
\begin{eqnarray*}
{and}~~~\eta_i^G &=& 0 ~~\forall~i \in I_{+-}(x^\ast) \cup I_{0-}(x^\ast) \cup I_{0+}(x^\ast),~~\eta_i^G  \geqslant 0~\forall i \in I_{+0}(x^\ast) \cup I_{00}(x^\ast),\\
\eta_i^H &=& 0 ~~\forall~i \in I_+(x^\ast),~~\eta_i^H  \geqslant  0~\forall ~i \in I_{0-}(x^\ast)~~ {\rm and}~~\eta_i^H ~ {\rm is}~ { \rm free}~ \forall~ i \in I_{0+}(x^\ast),\\
\eta_i^H  \eta_i^G &=& 0~~\forall~ i \in I_{00}(x^\ast).
\end{eqnarray*}
	\end{Theorem}
	\begin{proof}
	We consider MPVC in an equivalent form, called EMPVC  defined as follows:
\begin{eqnarray*}
\label{eq1.1}
  {\rm min} ~~ f(x)\\
 {\rm s.t.}~~ g(x) &\leqslant&  0, \\
 h(x) &=&  0, \\
H(x)-y &=&  0, \\
G(x)-z &=&  0,\\
 (x,y,z)&\in& \Omega
	\end{eqnarray*}
where~$\Omega = \{(x,y,z) \in \mathbb{R}^{n+2q}~|~y_i \geqslant 0,~  {z_i} y_i \leqslant 0  ,~~\forall i=1,...,q \}$.\\\\
Clearly, $\Omega$ is a nonempty closed set. It is an optimization problem with equalities, inequalities and an abstract set constraint, suppose we have a local minimum for this problem at $(x^\ast, ~y^\ast, ~z^\ast)$, then
		$$~~~~~~~~~~~~~~~~~~~~~~~~~y^\ast ~=~H(x^\ast)~~~~~~{\rm and}~~~~~~z^\ast ~=~G(x^\ast).$$
Applying the Lagrange multiplier rule obtained by Mordukhovich \cite[Theorem 5.11]{mordukhovich} (see also \cite[Corollary 6.15]{Rocka}), we have $\alpha \geqslant 0,~ (\alpha, \lambda, \mu, \eta^H,\eta^G)\neq 0 $ ~such that
		 \begin{eqnarray*}
		 	\label{eq:F-N-cone}
		 	\left( \begin{array}{ccc}
		 		0 \\0 \\0 \end{array}\right)
		 	\nonumber  & =   &   \alpha \left( \begin{array}{ccc}
		 		\nabla f(x^\ast)\\0\\0
		 	\end{array}\right)  + \sum_{i\in I_g(x^\ast)}\lambda_i \left( \begin{array}{ccc} \nabla g_i (x^\ast)\\0\\0 \end{array}\right)\\
		 	\nonumber  & + &  \sum_{i=1}^{p} \mu_i \left( \begin{array}{ccc}
		 		\nabla h_i(x^\ast)\\0\\0 \end{array}\right) - \sum_{i=1}^{q} \eta_i^H\left( \begin{array}{ccc} \nabla H_i(x^\ast)\\-e_i\\0 \end{array}\right)\\
		 	\textbf{}\nonumber  & - &   \sum_{i=1}^{q} \eta_i^G \left( \begin{array}{ccc} \nabla G_i(x^\ast)\\0\\-e_i \end{array}\right) +  \left( \begin{array}{ccc}
		 		0\\\xi\\\zeta \end{array}\right)   \,\,\,\,\,\,\
		 	\\
		 \end{eqnarray*}
		 and
		 \begin{equation*}
		  \lambda_{i}~\geq~0 ~~\forall ~ i \in I_g(x^\ast)
		 \end{equation*}
		 where $(0, \xi , \zeta)^T \in N_\Omega (x^\ast, y^\ast, z^\ast)$, the limiting normal cone of $\Omega$ at $(x^\ast, y^\ast, z^\ast)$ and  is given as
		 \begin{equation*}
		 N_\Omega (x^\ast, y^\ast, z^\ast) = \left\{ \left( \begin{array}{ccc}
		 0\\ \xi\\ \zeta
		 \end{array}\right) \Bigg | \begin{array}{ccc}
		 \xi_i = ~0~ =\zeta_i &;& {\rm if} ~ y^\ast_i >0,~z^\ast_i <0,\\
		 \xi_i =0 ,~ \zeta_i \geqslant 0 &;& {\rm if} ~y^\ast_i > 0, ~ z^\ast_i=0, \\
		  \zeta_i \geqslant 0, \xi_i \cdot \zeta_i = 0 &;&{\rm if} ~y^\ast_i ~=~0~= ~z^\ast_i,\\
		 \xi_i \leqslant 0 ,~ \zeta_i = 0 &;&{\rm if} ~y^\ast_i = 0, ~z^\ast_i <0,\\
		 \xi_i \in \mathbb{R} ,~ \zeta_i = 0 &;&{\rm if} ~y^\ast_i = 0, ~z^\ast_i > 0
		 \end{array}
		 \right \},
		 \end{equation*}
		 see \cite[Lemma 4.1]{hoheisel} and \cite[Lemma 3.2]{hoheisel 2} for the detailed computation. Hence, we obtain
		 \begin{equation*}
		 0~=~\alpha \nabla f(x^\ast) + \sum_{i \in I_g}^{} \lambda_i \nabla g_i(x^\ast) + \sum_{i=1}^{p} \mu \nabla h_i(x^\ast) - \sum_{i=1}^{q} \eta_i^H \nabla H_i(x^\ast) - \sum_{i=1}^{q} \eta_i^G \nabla G_i(x^\ast)
		 \end{equation*}
		 and
		 \begin{equation*}
		 0~=~\eta_i^H + \xi_i ~~\Rightarrow~~\eta_i^H ~=~-\xi_i,
		 \end{equation*}
		 \begin{equation*}
		 0~=~\eta_i^G + \zeta_i ~~\Rightarrow~~\eta_i^G ~=~-\zeta_i
		 \end{equation*}
		 now restriction on $\xi , ~\zeta$ in limiting normal cone yields the required result for multipliers, i.e. $\eta_i^H = 0 ~~\forall~i \in I_{+-}(x^\ast) \cup I_{+0}(x^\ast),~~~\eta_i^G = 0~~\forall~i \in I_{+-}(x^\ast) \cup I_{0-}(x^\ast) \cup I_{0+}(x^\ast) ~~{\rm and}~ \eta_i^G \leqslant 0 ~~\forall i \in I_{+0}(x^\ast),~~~\eta_i^H \geqslant 0~~\forall ~i \in I_{0-}(x^\ast)$ and for all $i \in I_{00}(x^\ast),~~\eta_i^G \leqslant 0$, $\eta_i^G \eta_i^H = 0$. Taking $-\eta_i^G$ in place of $\eta_i^G$, we get the desired conditions (i) and (ii) of the Theorem.
		\end{proof}
The proof of the following Theorem follows from a simple corollary of Theorem 4.2 established later in section 4. We have written this Theorem here to provide relevant results at one place.	
			\begin{Theorem}
\label{thm3.2}
				\textbf{A Fritz John type S-stationary condition:} Let $x^\ast$ be a local minimum of MPVC, then there exist  $\alpha \geqslant 0~ {\rm and} ~\lambda, ~\mu, ~\eta^H, ~\eta^G$, not all zero, such that\\
				$ {\rm (i)} ~~\alpha \nabla f(x^\ast) ~+~ \sum_{i=1}^{m} \lambda_i \nabla g_i(x^\ast) ~+~ \sum_{i=1}^{p} \mu_i \nabla h_i (x^\ast) ~+~ \sum_{i=1}^{q} \eta^G_i \nabla G_i(x^\ast)
				~-~ \sum_{i=1}^{q} \eta^H_i \nabla H_i(x^\ast)~=~0 $;\\\\
				${\rm (ii)}~~~ \lambda_i \geqslant 0 ~~\forall i \in I_g(x^\ast),~~
 		\lambda_i=0 ~\forall i \notin I_g(x^\ast)$
				\begin{eqnarray*}
 	{and}~~~\eta_i^G &=& 0 ~~~~\forall~i \in I_{+-}(x^\ast) \cup I_{0-}(x^\ast) \cup I_{0+}(x^\ast),~~\eta_i^G  \geqslant 0~\forall i \in I_{+0}(x^\ast) \cup I_{00}(x^\ast),\\
\eta_i^H &=& 0 ~~~~\forall~i \in I_+(x^\ast),~~\eta_i^H  \geqslant  0~\forall ~i \in I_{0-}(x^\ast)~~ {\rm and}~~\eta_i^H ~ {\rm is}~ { \rm free}~ \forall~ i \in I_{0+}(x^\ast),\\
 	\eta_i^H &\geqslant& 0,~\eta_i^G = 0  ~~~~\forall~i \in I_{00}(x^\ast).
 \end{eqnarray*}
\end{Theorem}
		
		Now, using this result we can establish some-well known results of MPEC for MPVC also. As in \cite[Corollary 2.1, Proposition 2.1]{ye 3}, Ye showed that any local optimizer of MPEC also becomes M-stationarity under (NNAMCQ or MPEC-GMFCQ). Here, first,  we prove M-stationarity  for MPVC under the constraint qualification MPVC-MFCQ. Although this result has been established for MPVC in \cite[Corollary 5.3]{hoheisel} under the exact penalty condition at the local minimizer, we are relaxing this exactness condition and using a different approach to prove the  result, which is independent and easier than \cite{hoheisel}. Indeed, our proof also shows that MPVC-MFCQ need not provide S-stationarity condition.
		\begin{Theorem}
			\textbf{M-stationary conditions:}
			Suppose  MPVC-MFCQ holds at a local minimizer $x^\ast$ of MPVC. Then $x^\ast$ will be an M-stationary point.
		\end{Theorem}
		\begin{proof}
			Since MPVC-MFCQ holds at a local minimizer $x^\ast$, this implies that there exists a vector $d \in \mathbb{R}^n$ such that
			\begin{eqnarray*}
				\nabla h_i (x^\ast)^T d &=&0 ~~~~~\forall~i = 1,...,p,\\
				\nabla H_i (x^\ast)^T d &=&0 ~~~~~\forall~i \in I_{00}(x^\ast) \cup I_{0+} (x^\ast),\\
				- \nabla g_i (x^\ast)^T d &>&0 ~~~~~\forall~i \in I_g(x^\ast),\\
				- \nabla G_i (x^\ast)^T d &>&0 ~~~~~\forall~i \in I_{+0}(x^\ast) \cup I_{00} (x^\ast),\\
				\nabla H_i (x^\ast)^T d &>&0 ~~~~~\forall~i \in I_{0-}(x^\ast).
			\end{eqnarray*}
			Now, by Motzkin's theorem \\
			\begin{equation*}
			\sum_{i \in I_g (x^\ast)}^{} \lambda _i (-\nabla g_i(x^\ast)) + \sum_{i=1}^{p} \mu_i \nabla h_i (x^\ast) + \sum_{i \in I_{+0}(x^\ast) \cup I_{00}(x^\ast)}^{} \eta_i^G (-\nabla G_i(x^\ast)) + \sum_{i \in I_{00}(x^\ast) \cup I_{0+}(x^\ast) \cup I_{0-}(x^\ast)}^{} \eta_i^H \nabla H_i(x^\ast) ~=~0
			\end{equation*}
has no non zero solution, where		
\begin{eqnarray}
\label{eq:nnm}
 \left\{\begin{array}{ccc}
                \lambda_i &\geqslant 0& \forall ~i \in I_g(x^\ast) ,\lambda_i = 0 ~\forall ~i \notin I_g(x^\ast),  \\ \\
                 \eta_i^G  &\geqslant 0& \forall~ i \in I_{+0}(x^\ast) \cup I_{00}(x^\ast), \mu_i \in \mathbb R, \\ \\
                 \eta_i^H &\geqslant 0& \forall~ i \in I_{0-}(x^\ast), \eta_i^H \in \mathbb R~ \forall~ i \in I_{00}(x^\ast) \cup I_{0+}(x^\ast)
                \end{array} \right\}
\end{eqnarray}

			that is, we can say that
			\begin{equation*}
			\sum_{i \in I_g(x^\ast)}^{} \lambda _i \nabla g_i(x^\ast) - \sum_{i=1}^{p} \mu_i \nabla h_i (x^\ast) + \sum_{i \in I_{+0}(x^\ast) \cup I_{00}(x^\ast)}^{} \eta_i^G \nabla G_i(x^\ast) - \sum_{i \in I_{00}(x^\ast) \cup I_{0+}(x^\ast) \cup I_{0-}(x^\ast)}^{} \eta_i^H \nabla H_i(x^\ast) ~=~0
			\end{equation*}
			 has no non zero solution satisfying the multiplier conditions (\ref{eq:nnm}).\\\\ Since $\mu$ is free, therefore
			 \begin{equation}
			 \sum_{i \in I_g(x^\ast)}^{} \lambda _i \nabla g_i(x^\ast) + \sum_{i=1}^{p} \mu_i \nabla h_i (x^\ast) + \sum_{i \in I_{+0}(x^\ast) \cup I_{00}(x^\ast)}^{} \eta_i^G \nabla G_i(x^\ast) - \sum_{i \in I_{00}(x^\ast) \cup I_{0+}(x^\ast) \cup I_{0-}(x^\ast)}^{} \eta_i^H \nabla H_i(x^\ast) ~=~0
			  \label{eq1}
			 \end{equation}
	 		 also has no nonzero solution with the restriction on the multipliers given in (\ref{eq:nnm}).\\
			 \hspace*{0.5 cm} Now, since $x^\ast$ is local minimizer, therefore by the Fritz John type M-stationary condition, there exist a nonzero multiplier $(\alpha, \lambda, \mu, \eta^H, \eta^G)$ such that\\\\
 $  \alpha \nabla f(x^\ast) + \sum_{i=1}^{m} \lambda_i \nabla g_i(x^\ast) + \sum_{i=1}^{p} \mu_i \nabla h_i (x^\ast) + \sum_{i=1}^{q} \eta^G_i \nabla G_i(x^\ast)
			 - \sum_{i=1}^{q} \eta^H_i \nabla H_i(x^\ast) = 0 $\\\\
where
\begin{eqnarray*}			
\lambda_i & \geqslant &  0 ~\forall  i \in I_g(x^\ast),~ \lambda_i = 0 ~\forall  i  \notin I_g (x^\ast)~~{\rm and}~~ \eta_i^H = 0 ~\forall i \in I_{+}(x^\ast),\\
\eta_i^G &=& 0~~\forall~i \in I_{+-}(x^\ast) \cup I_{0-}(x^\ast) \cup I_{0+}(x^\ast),\\
{\rm and}~~ \eta_i^G & \geqslant & 0 ~~\forall i \in I_{+0}(x^\ast),~~\eta_i^H \geqslant 0~~\forall ~i \in I_{0-}(x^\ast) ~{\rm and} ~ \eta_i^G \geqslant 0, \eta_i^G \eta_i^H = 0~ \forall i \in I_{00}(x^\ast).
\end{eqnarray*}
		\hspace*{0.5 cm}	 Here if $\alpha = 0$, then it contradicts that equation (\ref{eq1}), obtained by MPVC-MFCQ, has no nonzero solution with restriction on multipliers given in (\ref{eq:nnm}). Hence,  $\alpha \neq 0$ and then by proper scaling in Fritz John type optimality conditions we  have M-stationary conditions at $x^\ast$.
\end{proof}
\begin{Remark}
 The above M-stationarity is also derived in \cite [Theorem 3.4]{hoheisel 2} under MPVC-GCQ (MPVC-\emph{Guignard constraint qualification}), which is much weaker than MPVC-MFCQ. In the above proof, if we use the Fritz John type S-stationarity, then for $\alpha=0$, we may obtain nonzero solutions of (\ref{eq1}) with restriction on multipliers given in (\ref{eq:nnm}). For illustration, take $\eta_i^H > 0~ \forall i \in I_{00}$ and other multipliers zero, then they satisfy the KKT type S-stationary conditions, and also (\ref{eq1}) has a nonzero solution with the restriction given in (\ref{eq:nnm}). This contradicts MPVC-MFCQ; hence under MPVC-MFCQ the S-stationarity need not hold.
\end{Remark}

Now our next result shows that  the M-stationarity holds under a constraint qualification MPVC-GMFCQ, which is weaker than MPVC-MFCQ. Hence by using the next theorem and the fact that MPVC-MFCQ implies MPVC-GMFCQ, Theorem 3.3 is obvious. But we have included its proof to conclude the fact that MPVC-MFCQ is not sufficient for any local minimizer to be S-stationary.			
			\begin{Theorem}
				If $x^\ast$ is a local minimizer of MPVC and MPVC-GMFCQ holds at $x^\ast$, then $x^\ast$ is  an M-stationary point.
			\end{Theorem}
			\begin{proof}
			If $x^\ast$ is local minimizer, then we have by Fritz John type M-stationary condition, there exist  $\alpha \geqslant 0~ {\rm and} ~\lambda, ~\mu, ~\eta^H, ~\eta^G$ such that\\\\
			$ ~~\alpha \nabla f(x^\ast) ~+~ \sum_{i=1}^{m} \lambda_i \nabla g_i(x^\ast) ~+~ \sum_{i=1}^{p} \mu_i \nabla h_i (x^\ast) ~+~ \sum_{i=1}^{q} \eta^G_i \nabla G_i(x^\ast)
			~-~ \sum_{i=1}^{q} \eta^H_i \nabla H_i(x^\ast)~=~0 $\\\\
			${\rm where}~~~~~~~~\lambda_i \geqslant 0 ~~~\forall~i \in I_g(x^\ast),~~~~~~\lambda_i = 0 ~~~\forall~i \notin I_g(x^\ast)$
\begin{eqnarray*}
 		{and}~~~\eta_i^G &=& 0 ~~~~\forall~i \in I_{+-}(x^\ast) \cup I_{0-}(x^\ast) \cup I_{0+}(x^\ast),~~\eta_i^G  \geqslant 0~\forall i \in I_{+0}(x^\ast) \cup I_{00}(x^\ast),\\
\eta_i^H &=& 0 ~~~~\forall~i \in I_+(x^\ast),~~\eta_i^H  \geqslant  0~\forall ~i \in I_{0-}(x^\ast)~~ {\rm and}~~\eta_i^H ~ {\rm is}~ { \rm free}~ \forall~ i \in I_{0+}(x^\ast),\\
 		\eta_i^G . \eta_i^H &=&0 ~~~~\forall~i \in I_{00}(x^\ast).
 	\end{eqnarray*}
			If $\alpha = 0 $, then it violates the MPVC-GMFCQ. Hence $\alpha \neq 0$ and then by proper scaling we get M-stationarity at $x^\ast$.
			
			\end{proof}
			
			By using EMPVC form of MPVC, we can establish  M-stationary conditions for MPVC. The corresponding result has also been established for OPVIC in \cite[Corollary 4.8]{ye 2} and for MPEC in \cite{ye 3}. Here, we are using similar approach of the error bound as in \cite{ye 2,ye 3} to prove our result. But, firstly  analogous to the optimization problem with variational inequality constraints (OPVIC) \cite[Definition 4.1]{ye 2}, we  define local error bound for the constraint system of EMPVC  as follows.
			\begin{Definition}
				\textbf{Local error bound property for EMPVC:}\\
				The system of constraints
				\begin{eqnarray*}
					g(x)&\leqslant &0 ,\\
					h(x)&=&0,\\
					H(x) - y &=&0,\\
					G(x) - z &=&0,\\
					(y,z)&\in&\Omega
				\end{eqnarray*}
				where
				\begin{equation*}
				\Omega ~=~ \{(y,z) \in \mathbb{R}^{2q} ~|~y_i \geqslant 0,~ y_i z_i \leqslant 0, ~\forall ~i = 1,...,q \},
				\end{equation*}
				is said to have a local error bound at $(x^\ast, y^\ast, z^\ast)$, if there exist $\alpha, \delta, \epsilon > 0$ such that
				\begin{equation*}
				d((x,y,z), \Psi (0,0,0,0)) \leqslant \alpha ||(p,q,r,s)||
				\end{equation*}
				for all $(p,q,r,s) \in \mathbb{B}_\epsilon (0,0,0,0)$ and all $(x,y,z) \in \Psi (p,q,r,s) \cap \mathbb{B}_\delta (x^\ast,y^\ast,z^\ast)$, where
				\begin{equation*}
				\Psi (p,q,r,s) = \{(x,y,z) \in  \mathbb{R}^n \times \Omega | g(x) + p \leqslant 0; h(x) + q = 0 ; H(x)-y+s = 0 ; G(x)-z+r = 0\}
				\end{equation*}
			\end{Definition}
			Usually, in case when all the constraints are affine, KKT necessary optimality conditions hold without any additional constraint qualification, but we can not assure it for MPVC as it is a special and more difficult class of optimization problem. Therefore, to prove our next result, we define one more constraint qualification besides the above local error bound property for EMPVC, which is similar to \emph{linear constraint qualification} of MPEC case \cite[Definition 2.12]{ye 3}
			
			\begin{Definition}
			\textbf{MPVC-linear constraint qualification:} For the MPVC  problem, MPVC-linear constraint qualification is said to be satisfied if all the functions $g,h,G,H$ are affine.
			\end{Definition}
			Using these two constraint qualifications, we have the following result.
			\begin{Theorem}
				 \textbf{M-stationary conditions:}
				Let $z^\ast$ be a local optimal solution for MPVC, where all functions are continuously differentiable at $z^\ast$. If either MPVC-GMFCQ or MPVC-linear constraint qualification is satisfied at $z^\ast$, then $z^\ast$ is an M-stationary point.
			\end{Theorem}
			\begin{proof}
				Under MPVC-GMFCQ result holds obviously. For the later case, MPVC can be written equivalently as (EMPVC):
				\begin{eqnarray}
					\nonumber \min f(x) &&\\
					\nonumber {\rm s.t.}~~g(x) &\leqslant&0,\\
					\nonumber h(x) &=&0,\\
					\nonumber H(x) - y &=&0,\\
					\nonumber G(x) - z &=&0,\\
					(y,z) & \in & \Omega
				\end{eqnarray}
				 where
				\begin{equation*}
				\Omega ~=~ \{(y,z) \in \mathbb{R}^{2q} ~|~y_i \geqslant 0,~ y_i z_i \leqslant 0, ~\forall ~i = 1,...,q \}.
				\end{equation*}
Now, we consider the set of solutions to the perturbed constraints system for EMPVC.
				\begin{equation*}
				\Psi (p,q,r,s) = \{(x,y,z) \in \mathbb{R}^n \times \Omega | g(x) + p \leqslant 0 ; h(x) + q = 0 ; H(x)-y+s = 0 ; G(x)-z+r = 0\}
				\end{equation*}
				Since all $g(x), h(x), H(x), G(x)$ are affine, therefore graph of the set valued map $\Psi (p,q,r,s)$ is a union of polyhedral convex sets and hence $\Psi (p,q,r,s)$ is a polyhedral multifunction. By \cite[Proposition 1]{robinson} $\Psi$ is locally upper Lipschitz at each point of $\mathbb{R}^{2q+m+p}$, in particular  at $(0,0,0,0) \in \mathbb{R}^{2q+m+p}$, \emph{i.e.}  there is a neighbourhood $U$ of $(0,0,0,0)$ and $\alpha \geqslant 0$ such that \\
				\begin{equation*}
				\Psi (p,q,r,s) \subseteq \Psi (0,0,0,0) + \alpha ||(p,q,r,s)|| cl \mathbb{B} ~~~\forall~(p,q,r,s) \in U
				\end{equation*}
				 where $cl \mathbb{B}$ denotes the closed unit ball. Equivalently, the constraint system of EMPVC has a local error bound, i.e.\\
				 \begin{equation*}
				 d((x,y,z), \Psi (0,0,0,0)) ~\leqslant~\alpha ||(p,q,r,s)||
				 \end{equation*}
				 for all $(p,q,r,s) \in U$ and $(x,y,z) \in \Psi (p,q,r,s)$. Now by Clarke's principle of exact penalization \cite[Proposition 2.4.3]{clarke} $(x^\ast, y^\ast, z^\ast)$ is also a local optimal solution to the unconstrained problem \\
				 \begin{equation*}
				 \min f(x) + \mu_f d((x,y,z), \Psi (0,0,0,0))
				 \end{equation*}
				 Hence by the local error bound property and the fact that $x^\ast$ is  local optimal solution to MPVC, we have that $(x^\ast, 0,0,0)$ is a local optimal solution to the following
				 \begin{eqnarray*}
				 \min f(x) + \mu_f \alpha||(p,q,r,s)||\\
					{\rm s.t.} ~~g(x) + p  \leqslant 0,&& \\
				 	h(x) + q = 0,&& \\
				 	H(x) + s \geqslant 0,&& \\
				 	(G_i(x)+r)(H_i(x)+s) \leqslant 0 &&\forall~i=1,...,q,
				 \end{eqnarray*}
				  then at $(x^\ast, 0,0,0,0)$, MPVC-GMFCQ (or MPVC-NNAMCQ) is satisfied for the above problem. Suppose not, then there is a nonzero multiplier $(\lambda, \mu, \eta^G, \eta^H)$ such that
				  \begin{equation*}
				  0 = \sum_{i=1}^{m} \lambda_i \nabla (g_i(x^\ast)+p) + \sum_{i=1}^{p} \mu_i \nabla (h_i(x^\ast)+q) - \sum_{i=1}^{q} \eta_i^H \nabla(H_i(x^\ast) + s) + \sum_{i=1}^{q} \eta_i^G \nabla(G_i(x^\ast) + r)
				  \end{equation*}
				  That is,
				  \begin{eqnarray*}
				  	0 = \sum_{i=1}^{m} \lambda_i \nabla g_i(x^\ast) &+& \sum_{i=1}^{p} \mu_i \nabla h_i(x^\ast) - \sum_{i=1}^{q} \eta_i^H \nabla H_i(x^\ast) + \sum_{i=1}^{q} \eta_i^G \nabla  G_i(x^\ast) \\
				  	&+& \sum_{i=1}^{m} \lambda_i \nabla p + \sum_{i=1}^{p} \mu_i \nabla q - \sum_{i=1}^{q} \eta_i^H \nabla s + \sum_{i=1}^{q} \eta_i^G \nabla r.
				  \end{eqnarray*}
				Which implies that   $  \lambda_i ~=~\mu_i ~=~ \eta_i^H ~=~\eta_i^G ~=~0 $ for all $i$. Hence, MPVC-GMFCQ is satisfied at $(x^\ast, 0,0,0,0)$ for the above problem. Now in $(x^\ast, 0,0,0,0)$, $x$ component is same as optimal solution of original MPVC (\ref{eq1.2}). Therefore, MPVC-GMFCQ is also satisfied for original MPVC at $x^\ast$. Hence, $\alpha > 0$, and then by scaling, M-stationary conditions can be seen.
\end{proof}
		
	\section{ Enhanced  Stationarity and Weak Constraint Qualifications}
	\label{efjsc}
	Here, we present  strong versions of Fritz John type stationary conditions for the MPVC. In this section, our first  result is a refinement of the Fritz John type M-stationary conditions for the MPVC,  established in section \ref{fjsc}, in the sense that it provides information about a neighbourhood of extremal points in terms of converging sequence towards it. So, this result may be referred to as enhanced Fritz John type result for MPVC. These enhanced Fritz John conditions have already been established for several smooth and nonsmooth optimization problems in \cite{bert,bert2,bertsekas2,hestenes,Kanzow1,ye 1}.\\
	\hspace*{0.5 cm} The present result is motivated by \cite[Theorem 3.1]{Kanzow1} established for MPEC. In  \cite{Kanzow1}, it was stressed that their proof was completely elementary since they used Fr\'echet normal cone instead of limiting normal cone \cite{ye 3}. We follow \cite{Kanzow1},  using limiting and Fr\'echet normal cone respectively,  we obtain M-stationary and S-stationary type necessary optimality conditions for MPVC. We refer to Hoheisel and Kanzow  \cite{hoheisel 2} for detailed study of these generalized cones for MPVC.
	\begin{Theorem}
		\label{Theorem 4.1}
	\textbf{Enhanced Fritz John type M-stationary conditions:}
	\label{thmefj-M}
	Let $x^\ast$ be a local minimum of MPVC (P), then there exist multipliers $\alpha, ~\lambda, ~\mu, ~\eta^H, ~\eta^G$ such that\\
	$ {\rm (i)} ~~\alpha \nabla f(x^\ast) + \sum_{i=1}^{m} \lambda_i \nabla g_i(x^\ast) + \sum_{i=1}^{p} \mu_i \nabla h_i (x^\ast) + \sum_{i=1}^{q} \eta^G_i \nabla G_i(x^\ast)
	-\sum_{i=1}^{q} \eta^H_i \nabla H_i(x^\ast)~=~0 $;\\
	$ {\rm (ii)}~~ \alpha \geqslant 0,~~
		\lambda_i \geqslant ~~\forall~i \in I_g(x^\ast),~~~
		\lambda_i =0~~\forall~i \notin I_g(x^\ast)$
	\begin{eqnarray*}
	{and}~~~\eta_i^G &=& 0 ~~~~\forall~i \in I_{+-}(x^\ast) \cup I_{0-}(x^\ast) \cup I_{0+}(x^\ast),~~\eta_i^G  \geqslant 0~\forall i \in I_{+0}(x^\ast) \cup I_{00}(x^\ast)\\
\eta_i^H &=& 0 ~~~~\forall~i \in I_+(x^\ast),~~\eta_i^H  \geqslant  0~\forall ~i \in I_{0-}(x^\ast)~~ {\rm and}~~\eta_i^H ~ {\rm is}~ { \rm free}~ \forall~ i \in I_{0+}(x^\ast)\\
		\eta_i^G . \eta_i^H &=&0 ~~~~\forall~i \in I_{00}(x^\ast);
	\end{eqnarray*}
	$ {\rm (iii)} ~~\alpha ,~\lambda ,~\mu ,~\eta^G ,~\eta^H$ are not all equal to zero;\\
	$ {\rm (iv)}$ If $\lambda ,~\mu ,~\eta^G ,~\eta^H $ are not all equal to zero, then there is a sequence $\{x^k\} \rightarrow x^\ast $ such that $\forall k \in \mathbb{N}$, we have
	\begin{eqnarray*}
		f(x^k)&<&f(x^\ast),\\
		 \lambda_i &>&0  ~~\Rightarrow ~~\lambda_i g_i(x^k) >  0 ~~~ \{i=1,...,m\},\\
	 \mu_i &\neq&0  ~~\Rightarrow ~~\mu_i h_i(x^k) >  0 ~~~ \{i=1,...,p\},\\
	 \eta^H_i &\neq&0  ~~\Rightarrow ~~\eta^H_i H_i(x^k) <  0 ~~~ \{i=1,...,q\},\\
	 \eta^G_i &>&0  ~~\Rightarrow ~~\eta^G_i G_i(x^k) >  0 ~~~ \{i=1,...,q\}.	
	\end{eqnarray*}	
	\end{Theorem}
	\begin{proof}
	Firstly, we formulate MPVC  equivalently as  (EMPVC):
	\begin{eqnarray}
	\label{empvc2}
\nonumber {\rm min} ~~ f(x)\\
\nonumber {\rm s.t.}~~ g(x) &\leqslant&  0,\\
\nonumber h(x) &=&  0,\\
\nonumber	y - H(x) &=&  0,\\
\nonumber	z - G(x) &=&  0,\\
		(x,y,z)&\in&\Omega
\end{eqnarray}
where $\Omega~=~\{(x,y,z) \in \mathbb{R}^{n+q+q}~|~y_i \geqslant 0,~  y_i z_i\leqslant 0 ~~\forall i=1,...,q \}
$
	  is nonempty closed set, and we have a local minimum at $(x^\ast, ~y^\ast, ~z^\ast)$, then \\
	  $~~~~~~~~~~~~~~~~~~~~~~~~~~~~~~~y^\ast ~=~H(x^\ast)~~~~~~{\rm and}~~~~~~z^\ast ~=~G(x^\ast)$\\
	  Now using the idea of \cite[Proposition 2.1]{bertsekas2} , we choose $\epsilon > 0$ such that
	  \begin{equation*}
	  	f(x)~\geqslant~f(x^\ast)~~~\forall ~(x,~y,~z) \in \mathcal{S}
	  \end{equation*}
	   that are feasible for EMPVC, where
	  \begin{equation*}
	  \mathcal{S}~=~\{ (x,y,z) ~|~ \|(x,y,z)-(x^\ast,y^\ast,z^\ast)\| \leqslant \epsilon \}.
	  \end{equation*}
	  \hspace*{0.5 cm} Now consider the penalized problem approach given by McShane \cite{mcshane} and later elegantly used by Bertsekas \cite{bertsekas2}
	  \begin{eqnarray*}
{\rm min} ~F_k(x,y,z) \\
 {\rm s.t.}~ (x,y,z)\in \mathcal{S} \cap \Omega
\end{eqnarray*}
 with
 \begin{eqnarray*}
	  F_k(x,y,z)&:=&f(x)+\frac{k}{2} \sum_{i=1}^{m} (g_i^{+}(x^k))^2 +\frac{k}{2} \sum_{i=1}^{p} h_i(x)^2 + \frac{k}{2} \sum_{i=1}^{q} (y_i - H_i(x))^2 \\ && ~~~~~~~~ + \frac{k}{2} \sum_{i=1}^{q} (z_i - G_i(x))^2
	   + \frac{1}{2} ||(x,y,z) - (x^\ast, y^\ast, z^\ast)||^2 ~~\forall ~ k \in \mathbb{N}.
	  \end{eqnarray*}
	 Since $\mathcal{S} \cap \Omega $ is compact and $F_k$ is continuous, therefore this penalized problem has at least one solution say $(x^k, y^k, z^k) ~, ~\forall ~k \in \mathbb{N}$\\
	 \hspace*{0.5cm} Now we will show that this sequence $\{(x^k, y^k, z^k)\}$ converges to $(x^\ast, y^\ast, z^\ast)$. Note that
	 \begin{eqnarray*}
	 	F_k(x^k,y^k,z^k)&=&f(x^k)+\frac{k}{2} \sum_{i=1}^{m} (g_i^{+}(x^k))^2 +\frac{k}{2} \sum_{i=1}^{p} h_i(x^k)^2 + \frac{k}{2} \sum_{i=1}^{q} (y^k_i - H_i(x^k))^2 \\
		&&+ \frac{k}{2} \sum_{i=1}^{q} (z^k_i - G_i(x^k))^2 + \frac{1}{2} ||(x^k,y^k,z^k) - (x^\ast, y^\ast, z^\ast)||^2 \\
	 	&\leqslant& F_k (x^\ast, y^\ast, z^\ast)~=~ f(x^\ast)~~\forall ~ k \in \mathbb{N}.
	 \end{eqnarray*}
	 Since $\mathcal{S} \cap \Omega$ is compact, therefore sequence $\{f(x^k)\}$ is bounded. This yields
	 \begin{eqnarray*}
	 	\lim\limits_{k \rightarrow \infty} g_i^{+}(x^k) &=&0 ~~\forall ~i=1,...,m, \\
	 	\lim\limits_{k \rightarrow \infty} h_i(x^k) &=&0 ~~\forall ~i=1,...,p,\\
	 	\lim\limits_{k \rightarrow \infty} (y_i^k -  H_i(x^k)) &=&0~~\forall ~i=1,...,q,\\
	 	\lim\limits_{k \rightarrow \infty} (z^k_i - G_i(x^k)) &=&0 ~~\forall ~i=1,...,q.
	 \end{eqnarray*}
	 Otherwise, the left-hand side quantity of the above inequality would become unbounded and hence every accumulation point of $\{(x^k, y^k, z^k)\}$ is feasible for the reformulated MPVC (P).\\
	 \hspace*{0.5cm} The compactness of $\mathcal{S} \cap \Omega$  also ensures the existence of at least one accumulation point. Let $(\bar{x}, \bar{y}, \bar{z})$ be an arbitrary accumulation point of the sequence, then by the continuity  \\
	 \begin{equation*}
	 	f(\bar{x}) ~+~ \frac{1}{2} ||(\bar{x}, \bar{y}, \bar{z}) - (x^\ast, y^\ast, z^\ast)||^2 ~~\leqslant~~f(x^\ast)
	 \end{equation*}
	 and by the feasibility of $(\bar{x}, \bar{y}, \bar{z})$
	 \begin{equation*}
	 f(x^\ast)~\leqslant~f(\bar{x}).
	 \end{equation*}
	 Hence,
	 \begin{equation*}
	 	||(\bar{x}, \bar{y}, \bar{z}) - (x^\ast, y^\ast, z^\ast)||^2~=~0.
	 \end{equation*}
	 Hence, sequence $\{(x^k, y^k, z^k)\}$ converges to $\{(x^\ast, y^\ast, z^\ast)\}$. Now we may assume without loss of generality that $(x^k, y^k, z^k)$ is an interior point of $\mathcal{S} ~~\forall ~k \in \mathbb{N}$. Then by limiting subgradient version \cite[Theorem 3.2]{ye 2} of generalized Lagrange multiplier rule \cite[Theorem 6.1.1]{clarke}, we have
	 \begin{equation*}
	 - \nabla F_k(x^k, y^k, z^k)~\in~ N_\Omega (x^k, y^k, z^k)~~\forall~k \in \mathbb{N},
	 \end{equation*}
	 where the gradient of $F_k$ is given by
	
\begin{eqnarray*}
\label{eq:F-N-cone}
- \nabla F_k (x^k,y^k,z^k)
    \nonumber = & - &  \Biggr [ \left( \begin{array}{ccc}
	  \nabla f(x^k)\\0\\0
	  \end{array}\right)  + \sum_{i=1}^{m} k g_i^{+}(x^k) \left( \begin{array}{ccc} \nabla g_i (x^k)\\0\\0 \end{array}\right)\\
\nonumber  & + &  \sum_{i=1}^{p} k h_i(x^k) \left( \begin{array}{ccc}
	 \nabla h_i(x^k)\\0\\0 \end{array}\right) - \sum_{i=1}^{q} k (y_i^k - H_i(x^k))\left( \begin{array}{ccc} \nabla H_i(x^k)\\-e_i\\0 \end{array}\right)\\
\textbf{}\nonumber  & - &   \sum_{i=1}^{q} k (z_i^k - G_i(x^k))\left( \begin{array}{ccc} \nabla G_i(x^k)\\0\\-e_i \end{array}\right) + \left( \left( \begin{array}{ccc}
	  x^k\\y^k\\z^k \end{array}\right)- \left( \begin{array}{ccc}
	  x^\ast \\y^\ast \\z^\ast \end{array}\right) \right) \Biggr].
	  \end{eqnarray*}
 Hence, by using above gradient and limiting normal cone of $\Omega$ at $(x^k,y^k,z^k)$ as previous in Theorem 3.1, we obtain
	 \begin{eqnarray*}
	 \nabla f(x^k)&+&\sum_{i=1}^{m} k g_i^{+}(x^k) \nabla g_i(x^k)~ +~ \sum_{i=1}^{p} k h_i(x^k) \nabla h_i(x^k) ~-~ \sum_{i=1}^{q} k (y_i^k - H_i(x^k)) \nabla H_i(x^k)\\ &-& \sum_{i=1}^{q} k (z_i^k - G_i(x^k)) \nabla G_i(x^k) ~+~ (x^k - x^\ast)~=~0,
	 \end{eqnarray*}
	 \begin{eqnarray*}
	 k (y^k_i - H_i(x^k)) + (y_i^k -y^\ast) &=& -\xi_i,\\
	 k (z^k_i - G_i(x^k)) + (z_i^k -z^\ast) &=& -\zeta_i.
	 \end{eqnarray*}
	 Now we have\\
	 \hspace*{1.0cm} If $y_i^k >0,~z_i^k<0$ that is if $i \in I_{+-}(x^\ast)$ then
	 \begin{eqnarray*}
	 k (y_i^k - H_i(x^k))&=&-(y_i^k - y^\ast),\\
	 k (z_i^k - G_i(x^k))&=&-(z_i^k - z^\ast).
	 \end{eqnarray*}
	 \hspace*{1.0cm} If $y_i^k >0,~z_i^k = 0$ that is $i \in I_{+0}(x^\ast)$ then
	 \begin{eqnarray*}	
	  k (y_i^k - H_i(x^k))&=&-(y_i^k - y^\ast),\\
	  {\rm and}~~k (z^k_i - G_i(x^k)) + (z_i^k -z^\ast) &=& -\zeta~~\leq ~0,\\
	  k (z_i^k - G_i(x^k))&\leqslant&-(z_i^k - z^\ast).
	 \end{eqnarray*}
	 \hspace*{1.0cm} Similarly if $y_i^k ~=~0~=~z_i^k$ that is $i \in I_{00}(x^\ast)$ then
	 \begin{eqnarray*}
	 	k (z_i^k - G_i(x^k)) \leqslant -(z_i^k - z^\ast).
	 \end{eqnarray*}
	 \hspace*{1.0cm} If $y_i^k =0,~z_i^k < 0$ that is $i \in I_{0-}(x^\ast)$ then
	 \begin{eqnarray*}
	 	k (y_i^k - H_i(x^k))&\geqslant&-(y_i^k - y^\ast),\\
	 	k (z_i^k - G_i(x^k))&=&-(z_i^k - z^\ast).
	 \end{eqnarray*}
	 \hspace*{1.0cm} If $y_i^k =0,~z_i^k > 0$ that is $i \in I_{0+}(x^\ast)$ then
	 \begin{eqnarray*}
	 	k (y_i^k - H_i(x^k))~+~(y_i^k - y^\ast)&\in& \mathbb{R},\\
	 	k (z_i^k - G_i(x^k))&=&-(z_i^k - z^\ast).
	 \end{eqnarray*}
	 Now we define the multipliers,
	 \begin{equation*}
	 \delta_k = \left( 1 + \sum_{i=1}^{m}(k g_i^{+}(x^k))^2 + \sum_{i=1}^{p} (k h_i(x^k))^2 + \sum_{i=1}^{q}(k (y_i^k - H_i(x^k)))^2 + \sum_{i=1}^{q}(k (z_i^k - G_i(x^k)))^2\right)^{\frac{1}{2}}
	 \end{equation*}
	 and
	 \begin{eqnarray*}
	 \alpha_k &=&\frac{1}{\delta_k} \\
	 \lambda_i^k &=&\frac{k \max \{0, g_i(x^k)\}}{\delta_k}~~\forall ~~i=1,...,m,\\
	 \mu_i^k &=& \frac{k h_i(x^k)}{\delta_k} ~~\forall ~~i=1,...,p,\\
	 \eta^{H^k}_{i} &=& \frac{k (y_i^k - H_i(x^k))}{\delta_k} ~~\forall ~~i=1,...,q,\\
	 \eta^{G^k}_{i} &=& \frac{k (z_i^k - G_i(x^k))}{\delta_k} ~~\forall ~~i=1,...,q.
	 \end{eqnarray*}
	 Since $||(\alpha_k,  \lambda^k,  \mu^k,  \eta^{H^k}_{i},  \eta^{G^k}_{i})||= 1$ for all $k \in \mathbb{N}$. Hence, we may assume that this sequence of multipliers converges to some limit $(\alpha , \lambda , \mu , \eta^H , \eta^G)~\neq~0$.\\
	 Now we will analyze some properties of this limit. Since $\alpha_k \rightarrow \alpha$, therefore sequence $\{\delta_k\}$ either diverges to $+\infty$ or converges to some positive value (greater than or equal to one). By continuity of gradients and because of $x^k \rightarrow x^\ast$, we obtain
	 \begin{equation}
	 \alpha \nabla f(x^\ast) + \sum_{i=1}^{m} \lambda_i \nabla g_i(x^\ast) + \sum_{i=1}^{p} \mu_i \nabla h_i(x^\ast) - \sum_{i=1}^{q} \eta_i^G \nabla G_i(x^\ast) - \sum_{i=1}^{q} \eta_i^H \nabla H_i(x^\ast)~=~0.
	 \label{en m stat 1}
	 \end{equation}
	 Furthermore, $\alpha \geqslant 0$ and $\lambda_i \geqslant 0$ for all $i \in I_g(x^\ast)$, additionally we have $\lambda_i = 0$ for all $i \notin I_g(x^\ast)$. Now remember,
	 \begin{eqnarray*}
	  (y^\ast, z^\ast)~=~(H(x^\ast), G(x^\ast)) {~\rm and~}
	  (x^k, y^k, z^k)~\in ~\Omega~~\forall ~k \in \mathbb{N}.
	 \end{eqnarray*}
	 Now, for all $i \in I_{+0}(x^\ast) \cup I_{+-}(x^\ast)$, where $y_i^k >0, z_i^k =0$ or $y_i^k >0, z_i^k <0$, this yields
	 \begin{eqnarray*}
	 \eta^H_i &=& \lim\limits_{k \rightarrow \infty} \frac{k (y^k_i - H_i(x^k))}{\delta_k}\\
	 &=&\lim\limits_{k \rightarrow \infty } \frac{-k (y_i^k - y^\ast)}{\delta_k}~=~0.
	 \end{eqnarray*}
	 Similarly, for all $i \in I_{+-}(x^\ast) \cup I_{0-}(x^\ast) \cup I_{0+}(x^\ast)$, we have $\eta^G_i ~=~0$.\\
	  Now for all $i \in I_{0-}(x^\ast)$
	 \begin{eqnarray*}
	 \eta^H_i &=&\lim\limits_{k \rightarrow \infty} \frac{k (y_i^k - H_i(x^k))}{\delta_k}\\
	 &\geqslant&\lim\limits_{k \rightarrow \infty} \frac{- (y^k_i - y^\ast)}{\delta_k}~=~0
	 \end{eqnarray*}
	 and for all $i \in I_{+0}(x^\ast) \cup I_{00}(x^\ast)$,
	 \begin{eqnarray}
	 \eta^G_i &=&\lim\limits_{k \rightarrow \infty} \frac{k (z^k_i - G_i(x^k))}{\delta_k}\\
	 &\leqslant&\lim\limits_{k \rightarrow \infty} \frac{- (z_i^k - z^\ast)}{\delta_k}~=~0
	 \label{en m stat 2}
	 \end{eqnarray}
	 and  $ \forall i \in I_{00}(x^\ast)$, $\eta^G_i ~\leqslant~0$ and $ \eta^G_i  \eta^H_i ~=~0.$\\
 \hspace*{0.5cm} Replace $\eta_i^G$ by $-\eta_i^G$, the first negative term in eq (\ref{en m stat 1}) becomes positive and we get $(i)$, with $\eta_i^G \geqslant 0 ~\forall~i \in I_{+0}(x^\ast) \cup I_{00}(x^\ast)$ in condition (ii). \\
	 Finally, assume that $(\lambda, \mu, \eta^G, \eta^H) \neq 0$, then $(\lambda^k, \mu^k, \eta^G_k, \eta^H_k) \neq 0$ for all $k \in \mathbb{N}$ sufficiently large. Hence by the definition of multipliers, we have $(x^k, y^k, z^k) \neq (x^\ast, y^\ast, z^\ast)$.\\
	 Therefore, we have for all $k$ sufficiently large
	 \begin{equation*}
	 f(x^k) < f(x^k) + \frac{1}{2} ||(x^k, y^k, z^k) - (x^\ast, y^\ast, z^\ast)|| \leqslant f(x^\ast).
	 \end{equation*}
	 Hence,  $f(x^k) < f(x^\ast)$. Further, we have the following results for all $i$ and all $k$ sufficiently large
	 \begin{equation*}
	 \lambda_i > 0 ~\Rightarrow ~\lambda_i^k > 0 ~\Rightarrow~ g_i(x^k) > 0 ~\Rightarrow ~\lambda_i g_i(x^k) > 0,
	 \end{equation*}
	 \begin{equation*}
	 \mu_i \neq 0 ~\Rightarrow~\mu_i \mu_i^k > 0 ~\Rightarrow~\mu_i h_i(x^k) > 0.
	 \end{equation*}
	 Further, if $\eta^H_i \neq 0$ for $i \in \{1,...,q\}$, then
	 \begin{eqnarray*}
	 \eta^H_i \eta^H_{ik} &>&0,\\
	 \eta^H_i (y_i^k - H_i(x^k))&>&0,\\
	 \eta^H_i H_i(x^k)&<&\eta^H_i y_i^k.
	 \end{eqnarray*}
	 Since whenever $y_i^k > 0$ for infinitely many $k$, then $\eta_i^H = 0$. Hence, in our case $y_i^k = 0~ \forall~k$ sufficiently large, so
	 \begin{eqnarray*}
	 	\eta^H_i H_i(x^k)&<&0.
	 \end{eqnarray*}
	 Now if $\eta^G_i \neq 0$, then
	 \begin{eqnarray*}
	 	\eta^G_i \eta^G_{ik} &>&0,\\
	 	\eta^G_i (z_i^k - G_i(x^k))&>&0,\\
	 	\eta^G_i G_i(x^k)&<&\eta^G_i z_i^k.
	 \end{eqnarray*}
	 If $z_i^k \neq 0$ for infinitely many $k$, then $\eta_i^G = 0$, thus in our case $z_i^k = 0  $  for all sufficiently large $k$.\\
	 Hence
	 \begin{eqnarray*}
	 	\eta^G_i G_i(x^k)&<&0~~~{\rm for~ all}~k~{\rm sufficiently ~large},
	 \end{eqnarray*}
that is \begin{eqnarray*}
	 	-\eta^G_i G_i(x^k)&>&0~~~{\rm for~ all}~k~{\rm sufficiently ~large.}
	 \end{eqnarray*}
As mentioned earlier that we replace $\eta^G_i$ by $-\eta^G_i$. Hence, we obtain
\begin{eqnarray*}
	 	\eta^G_i >0 \Rightarrow \eta^G_i G_i(x^k)&>&0~~~{\rm for~ all}~k~{\rm sufficiently ~large.}
	 \end{eqnarray*}
This completes the proof of Theorem.
	\end{proof}
	\begin{Remark}
	The conditions in Theorem 3.1 and Theorem 4.1 are identical except condition {\rm (iv)} in Theorem 4.1. Hence,  Theorem 3.1 directly follows by the above Theorem, but we have included its proof  in the previous section  to provide an easy approach.
	\end{Remark}
 If we replace limiting normal cone with Fr\'echet normal cone in above result, then multipliers corresponding to biactive set will be changed and it yields an enhanced S-stationarity type condition. Fr\'echet normal cone of $\Omega$ at $(x,y,z)$ can be given as in \cite[Lemma 3.2]{hoheisel 2} by
	\begin{equation}
	\label{Frechet normal}
	N_\Omega^F (x, y, z) = \left\{ \left( \begin{array}{ccc}
	0\\ \xi\\ \zeta
	\end{array}\right) \Bigg | \begin{array}{ccc}
	\xi_i = ~0~ =\zeta_i &;& if ~ y_i >0,~z_i <0,\\
	\xi_i =0 ,~ \zeta_i \geqslant 0 &;&if ~y_i > 0, ~ z_i=0, \\
	\xi_i \leqslant 0 ,~\zeta_i = 0 &;&if ~y_i ~=~0~= ~z_i,\\
	\xi_i \leqslant 0 ,~ \zeta_i = 0 &;&if ~y_i = 0, ~z_i <0,\\
	\xi_i \in \mathbb{R} ,~ \zeta_i = 0 &;&if ~y_i = 0, ~z_i > 0
	\end{array}
	\right \}.
	\end{equation}
	\begin{Theorem}
		 \textbf{Enhanced Fritz John type S-stationary conditions:}
	Let $x^\ast$ be a local minimum of MPVC, then there exist multipliers $\alpha, ~\lambda, ~\mu, ~\eta^H, ~\eta^G$ such that \\
	$ {\rm (i)} ~\alpha \nabla f(x^\ast) + \sum_{i=1}^{m} \lambda_i \nabla g_i(x^\ast) + \sum_{i=1}^{p} \mu_i \nabla h_i (x^\ast) + \sum_{i=1}^{q} \eta^G_i \nabla G_i(x^\ast)
	-\sum_{i=1}^{q} \eta^H_i \nabla H_i(x^\ast)~=~0 $;\\
$ {\rm (ii)}~~ \alpha \geqslant 0,~~
		\lambda_i \geqslant ~~\forall~i \in I_g(x^\ast),~~~
		\lambda_i =0~~\forall~i \notin I_g(x^\ast)$
	\begin{eqnarray*}
	{and}~~~\eta_i^G &=& 0 ~~~~\forall~i \in I_{+-}(x^\ast) \cup I_{0-}(x^\ast) \cup I_{0+}(x^\ast),~~\eta_i^G  \geqslant 0~\forall i \in I_{+0}(x^\ast) \cup I_{00}(x^\ast),\\
\eta_i^H &=& 0 ~~~~\forall~i \in I_+(x^\ast),~~\eta_i^H  \geqslant  0~\forall ~i \in I_{0-}(x^\ast)~~ {\rm and}~~\eta_i^H ~ {\rm is}~ { \rm free}~ \forall~ i \in I_{0+}(x^\ast),\\
		\eta_i^H &\geqslant& 0,~\eta_i^G = 0 ~~~~\forall~i \in I_{00}(x^\ast);
	\end{eqnarray*}
	$ {\rm (iii)} ~~\alpha ,~\lambda ,~\mu ,~\eta^G ,~\eta^H$ are not all equal to zero;\\
	$ {\rm (iv)}$ If $\lambda ,~\mu ,~\eta^G ,~\eta^H $ are not all equal to zero, then there is a sequence $\{x^k\} \rightarrow x^\ast {\rm such~ that}~ \forall k \in \mathbb{N}$, we have
	\begin{eqnarray*}
		f(x^k)&<&f(x^\ast)\\
		 \lambda_i &>&0  ~~\Rightarrow ~~\lambda_i g_i(x^k) >  0 ~~~ \{i=1,...,m\},\\
		 \mu_i &\neq&0  ~~\Rightarrow ~~\mu_i h_i(x^k) >  0 ~~~ \{i=1,...,p\},\\
		 \eta^H_i &\neq&0  ~~\Rightarrow ~~\eta^H_i H_i(x^k) <  0 ~~~ \{i=1,...,q\},\\
		 \eta^G_i &>&0  ~~\Rightarrow ~~\eta^G_i G_i(x^k) >  0 ~~~ \{i=1,...,q\}.	
	\end{eqnarray*}	
		\end{Theorem}
		\begin{proof}
The idea of the proof is similar to that of the Theorem \ref{Theorem 4.1}. We follow the same lines of the proof of the Theorem \ref{Theorem 4.1} up to just before the generalized Lagrange multiplier rule using limiting normal cone. After that,  we write out  the  standard optimality condition in terms of the gradient of $F_k$ and Fr\'echet normal cone,		
			
			\begin{equation*}
			- \nabla F_k(x^k, y^k, z^k)~\in~ N_\Omega^F (x^k, y^k, z^k)~~\forall~k \in \mathbb{N},
			\end{equation*}
			 here gradient and Fr\'echet normal cone are borrowed   from Theorem \ref{Theorem 4.1} and eq (\ref{Frechet normal}) respectively, again we obtain
			\begin{eqnarray*}
				\nabla f(x^k)&+&\sum_{i=1}^{m} k g_i^{+}(x^k) \nabla g_i(x^k)~ +~ \sum_{i=1}^{p} k h_i(x^k) \nabla h_i(x^k) ~-~ \sum_{i=1}^{q} k (y_i^k - H_i(x^k)) \nabla H_i(x^k)\\ &-& \sum_{i=1}^{q} k (z_i^k - G_i(x^k)) \nabla G_i(x^k) ~+~ (x^k - x^\ast)~=~0,
			\end{eqnarray*}
			\begin{eqnarray*}
				k (y^k_i - H_i(x^k)) + (y_i^k -y^\ast) &=& -\xi_i,\\
				k (z^k_i - G_i(x^k)) + (z_i^k -z^\ast) &=& -\zeta_i.
			\end{eqnarray*}
			Now we have,\\
			\hspace*{1.0cm} if $y_i^k >0,~z_i^k<0$ that is if $i \in I_{+-}(x^\ast)$ then
			\begin{eqnarray*}
				k (y_i^k - H_i(x^k))&=&-(y_i^k - y^\ast),\\
				k (z_i^k - G_i(x^k))&=&-(z_i^k - z^\ast).
			\end{eqnarray*}
			\hspace*{1.0cm} If $y_i^k >0,~z_i^k = 0$ that is $i \in I_{+0}(x^\ast)$ then
			\begin{eqnarray*}	
				k (y_i^k - H_i(x^k))&=&-(y_i^k - y^\ast),\\
				{\rm and}~~k (z^k_i - G_i(x^k)) + (z_i^k -z^\ast) &=& -\zeta~~\leq ~0,\\
				k (z_i^k - G_i(x^k))&\leqslant&-(z_i^k - z^\ast).
			\end{eqnarray*}
			\hspace*{1.0cm} If $y_i^k =0,~z_i^k < 0$ or $y_i^k ~=~0~=~z_i^k$  that is $i \in I_{0-}(x^\ast) \cup I_{00}(x^\ast)$  then
			\begin{eqnarray*}
				k (y_i^k - H_i(x^k))&\geqslant&-(y_i^k - y^\ast),\\
				k (z_i^k - G_i(x^k))&=&-(z_i^k - z^\ast).
			\end{eqnarray*}
			\hspace*{1.0cm} If $y_i^k =0,~z_i^k > 0$ that is $i \in I_{0+}(x^\ast)$ then
			\begin{eqnarray*}
				k (y_i^k - H_i(x^k))~+~(y_i^k - y^\ast)&\in& \mathbb{R},\\
				k (z_i^k - G_i(x^k))&=&-(z_i^k - z^\ast).
			\end{eqnarray*}
			Now we define the multipliers,
			\begin{equation*}
			\delta_k = \left( 1 + \sum_{i=1}^{m}(k g_i^{+}(x^k))^2 + \sum_{i=1}^{p} (k h_i(x^k))^2 + \sum_{i=1}^{q}(k (y_i^k - H_i(x^k)))^2 + \sum_{i=1}^{q}(k (z_i^k - G_i(x^k)))^2\right)^{\frac{1}{2}}
			\end{equation*}
			and
			\begin{eqnarray*}
				\alpha_k &=&\frac{1}{\delta_k} \\
				\lambda_i^k &=&\frac{k \max \{0, g_i(x^k)\}}{\delta_k}~~\forall ~~i=1,...,m,\\
				\mu_i^k &=& \frac{k h_i(x^k)}{\delta_k} ~~\forall ~~i=1,...,p,\\
				\eta^{H^k}_{i} &=& \frac{k (y_i^k - H_i(x^k))}{\delta_k} ~~\forall ~~i=1,...,q,\\
				\eta^{G^k}_{i} &=& \frac{k (z_i^k - G_i(x^k))}{\delta_k} ~~\forall ~~i=1,...,q.
			\end{eqnarray*}
			Since $||(\alpha_k,  \lambda^k,  \mu^k,  \eta^{H^k}_{i},  \eta^{G^k}_{i})||= 1$ for all $k \in \mathbb{N}$. Hence, we may assume that this sequence of multipliers converges to some limit $(\alpha , \lambda , \mu , \eta^H , \eta^G)~\neq~0$.\\
			Now we will analyze some properties of this limit. Since $\alpha_k \rightarrow \alpha$, therefore sequence $\{\delta_k\}$ either diverges to $+\infty$ or converges to some positive value (greater than or equal to one). By continuity of gradients and because of $x^k \rightarrow x^\ast$, we obtain
			\begin{equation}
			\alpha \nabla f(x^\ast) + \sum_{i=1}^{m} \lambda_i \nabla g_i(x^\ast) + \sum_{i=1}^{p} \mu_i \nabla h_i(x^\ast) - \sum_{i=1}^{q} \eta_i^G \nabla G_i(x^\ast) - \sum_{i=1}^{q} \eta_i^H \nabla H_i(x^\ast)~=~0.
			\label{en m stat 1}
			\end{equation}
			Furthermore, $\alpha \geqslant 0$ and $\lambda_i \geqslant 0$ for all $i \in I_g(x^\ast)$, additionally we have $\lambda_i = 0$ for all $i \notin I_g(x^\ast)$. Now remember,
			\begin{eqnarray*}
				(y^\ast, z^\ast)~=~(H(x^\ast), G(x^\ast)) {~\rm and~}
				(x^k, y^k, z^k)~\in ~\Omega~~\forall ~k \in \mathbb{N}.
			\end{eqnarray*}
			Now, for all $i \in I_{+0}(x^\ast) \cup I_{+-}(x^\ast)$, where $y_i^k >0, z_i^k =0$ or $y_i^k >0, z_i^k <0$, this yields
			\begin{eqnarray*}
				\eta^H_i &=& \lim\limits_{k \rightarrow \infty} \frac{k (y^k_i - H_i(x^k))}{\delta_k}\\
				&=&\lim\limits_{k \rightarrow \infty } \frac{-k (y_i^k - y^\ast)}{\delta_k}~=~0.
			\end{eqnarray*}
			Similarly, for all $i \in I_{+-}(x^\ast) \cup I_{0-}(x^\ast) \cup I_{0+}(x^\ast) \cup I_{00}(x^\ast)$, we have $\eta^G_i ~=~0$.\\
			Now for all $i \in I_{0-}(x^\ast) \cup I_{00}(x^\ast)$
			\begin{eqnarray*}
				\eta^H_i &=&\lim\limits_{k \rightarrow \infty} \frac{k (y_i^k - H_i(x^k))}{\delta_k}\\
				&\geqslant&\lim\limits_{k \rightarrow \infty} \frac{- (y^k_i - y^\ast)}{\delta_k}~=~0
			\end{eqnarray*}
			and for all $i \in I_{+0}(x^\ast)$,
			\begin{eqnarray}
			\eta^G_i &=&\lim\limits_{k \rightarrow \infty} \frac{k (z^k_i - G_i(x^k))}{\delta_k}\\
			&\leqslant&\lim\limits_{k \rightarrow \infty} \frac{- (z_i^k - z^\ast)}{\delta_k}~=~0
			\label{en m stat 2}
			\end{eqnarray}
			that is  $ \forall i \in I_{00}(x^\ast)$, $\eta^G_i = 0$ and $\eta^H_i \geqslant 0.$\\
			\hspace*{0.5cm} Replace $\eta_i^G$ by $-\eta_i^G$, the first negative term in eq (\ref{en m stat 1}) becomes positive and we get $(i)$, with $\eta_i^G \geqslant 0 ~\forall~i \in I_{+0}(x^\ast)$ in condition (ii). Also $\eta_i^G = 0$ can be treated as $\eta_i^G \geqslant 0$ for biactive set.\\
			Finally, assume that $(\lambda, \mu, \eta^G, \eta^H) \neq 0$, then $(\lambda^k, \mu^k, \eta^G_k, \eta^H_k) \neq 0$ for all $k \in \mathbb{N}$ sufficiently large. Hence by the definition of multipliers, we have $(x^k, y^k, z^k) \neq (x^\ast, y^\ast, z^\ast)$.\\
			Therefore, we have for all $k$ sufficiently large
			\begin{equation*}
			f(x^k) < f(x^k) + \frac{1}{2} ||(x^k, y^k, z^k) - (x^\ast, y^\ast, z^\ast)|| \leqslant f(x^\ast).
			\end{equation*}
			Hence,  $f(x^k) < f(x^\ast)$. Further, we have the following results for all $i$ and all $k$ sufficiently large
			\begin{equation*}
			\lambda_i > 0 ~\Rightarrow ~\lambda_i^k > 0 ~\Rightarrow~ g_i(x^k) > 0 ~\Rightarrow ~\lambda_i g_i(x^k) > 0,
			\end{equation*}
			\begin{equation*}
			\mu_i \neq 0 ~\Rightarrow~\mu_i \mu_i^k > 0 ~\Rightarrow~\mu_i h_i(x^k) > 0.
			\end{equation*}
			Further, if $\eta^H_i \neq 0$ for $i \in \{1,...,q\}$, then
			\begin{eqnarray*}
				\eta^H_i \eta^H_{ik} &>&0,\\
				\eta^H_i (y_i^k - H_i(x^k))&>&0,\\
				\eta^H_i H_i(x^k)&<&\eta^H_i y_i^k.
			\end{eqnarray*}
			Since whenever $y_i^k > 0$ for infinitely many $k$, then $\eta_i^H = 0$. Hence, in our case $y_i^k = 0~ \forall~k$ sufficiently large, so
			\begin{eqnarray*}
				\eta^H_i H_i(x^k)&<&0.
			\end{eqnarray*}
			Now if $\eta^G_i \neq 0$, then
			\begin{eqnarray*}
				\eta^G_i \eta^G_{ik} &>&0,\\
				\eta^G_i (z_i^k - G_i(x^k))&>&0,\\
				\eta^G_i G_i(x^k)&<&\eta^G_i z_i^k.
			\end{eqnarray*}
			If $z_i^k \neq 0$ for infinitely many $k$, then $\eta_i^G = 0$, thus in our case $z_i^k = 0  $  for all sufficiently large $k$.\\
			Hence
			\begin{eqnarray*}
				\eta^G_i G_i(x^k)&<&0~~~{\rm for~ all}~k~{\rm sufficiently ~large},
			\end{eqnarray*}
			that is \begin{eqnarray*}
				-\eta^G_i G_i(x^k)&>&0~~~{\rm for~ all}~k~{\rm sufficiently ~large.}
			\end{eqnarray*}
			As mentioned earlier that we replace $\eta^G_i$ by $-\eta^G_i$. Hence, we obtain
			\begin{eqnarray*}
				\eta^G_i >0 \Rightarrow \eta^G_i G_i(x^k)&>&0~~~{\rm for~ all}~k~{\rm sufficiently ~large.}
			\end{eqnarray*}
			This completes the proof of Theorem.
		\end{proof}

	 It is necessary to mention here that Theorem \ref{thm3.2} is just a consequence of the above result, which we have stated in the previous section without proof. Indeed, we have the following
\begin{Corollary}
Let $x^\ast$ be a local minimum of MPVC, then there exist multipliers $\alpha, ~\lambda, ~\mu, ~\eta^H, ~\eta^G$, not all zero, such that the Fritz John type S-stationarity given in Theorem \ref{thm3.2} holds.
\end{Corollary}
			Now, we can define some enhanced stationary conditions associated with enhanced Fritz John type conditions, which we have derived above.
	 \begin{Definition}
	 	\textbf{ Enhanced M-stationary conditions for MPVC:} Let $x^\ast$ be a feasible point of MPVC. Then we say that the enhanced M-stationary condition holds at $x^\ast$ if and only if there are multipliers $(\lambda, \mu, \eta^G, \eta^H) \neq (0,0,0,0)$ such that \\
	 	${\rm (i)}~~ 0 = \nabla f(x^\ast)~+~\sum_{i =1}^{m} \lambda_i \nabla g_i(x^\ast)~+~ \sum_{i =1}^{p} \mu_i \nabla h_i(x^\ast)~+~\sum_{i=1}^{q} \eta_i^G \nabla G_i(x^\ast)~-~\sum_{i=1}^{q} \eta_i^H \nabla H_i(x^\ast)$,\\
	 	${\rm (ii)} ~~~\lambda_i \geqslant 0 ~~~\forall~i \in I_g(x^\ast),~
	 		\lambda_i =0 ~\forall~i \notin I_g(x^\ast)$
\begin{eqnarray*}	 		
{and}~~~\eta_i^G &=& 0 ~~~~\forall~i \in I_{+-}(x^\ast) \cup I_{0-}(x^\ast) \cup I_{0+}(x^\ast),~~\eta_i^G  \geqslant 0~\forall i \in I_{+0}(x^\ast) \cup I_{00}(x^\ast),\\
\eta_i^H &=& 0 ~~~~\forall~i \in I_+(x^\ast),~~\eta_i^H  \geqslant  0~\forall ~i \in I_{0-}(x^\ast)~~ {\rm and}~~\eta_i^H ~ {\rm is}~ { \rm free}~ \forall~ i \in I_{0+}(x^\ast),\\
\eta_i^G.\eta_i^H & = &0  ~~~~\forall~i \in I_{00}(x^\ast),
\end{eqnarray*}
	 	${\rm (iii)} $ If $ \lambda, \mu, \eta^G, \eta^H $ not all equal to zero, then there is a sequence $\{x^k\} \rightarrow x^\ast$ such that  $\forall ~k \in \mathbb{N}$,we have
	\begin{eqnarray*}
				 \lambda_i &>&0  ~~\Rightarrow ~~\lambda_i g_i(x^k) >  0 ~~~ \{i=1,...,m\},\\
		 \mu_i &\neq&0  ~~\Rightarrow ~~\mu_i h_i(x^k) >  0 ~~~ \{i=1,...,p\},\\
		 \eta^H_i &\neq&0  ~~\Rightarrow ~~\eta^H_i H_i(x^k) <  0 ~~~ \{i=1,...,q\},\\
		 \eta^G_i &>&0  ~~\Rightarrow ~~\eta^G_i G_i(x^k) >  0 ~~~ \{i=1,...,q\}.	
	\end{eqnarray*}	
	 \end{Definition}
\begin{Definition}
\textbf{Enhanced S-stationary conditions for MPVC:} Any feasible point $x^\ast$ of MPVC is said to satisfy enhanced S-stationary condition thereat if and only if there exist multipliers $(\lambda, \mu, \eta^G, \eta^H) \neq (0,0,0,0)$ such that \\
	${\rm (i)}~~ 0 = \nabla f(x^\ast)~+~\sum_{i =1}^{m} \lambda_i \nabla g_i(x^\ast)~+~ \sum_{i =1}^{p} \mu_i \nabla h_i(x^\ast)~+~\sum_{i=1}^{q} \eta_i^G \nabla G_i(x^\ast)~-~\sum_{i=1}^{q} \eta_i^H \nabla H_i(x^\ast)$,\\
	$ {\rm (ii)} ~~~\lambda_i \geqslant 0 ~~~\forall~i \in I_g(x^\ast),~~~\lambda_i = 0 ~~~\forall~i \notin I_g(x^\ast)$
\begin{eqnarray*}
{and}~~~\eta_i^G &=& 0 ~~~~\forall~i \in I_{+-}(x^\ast) \cup I_{0-}(x^\ast) \cup I_{0+}(x^\ast),~~\eta_i^G  \geqslant 0~\forall i \in I_{+0}(x^\ast) \cup I_{00}(x^\ast),\\
\eta_i^H &=& 0 ~~~~\forall~i \in I_+(x^\ast),~~\eta_i^H  \geqslant  0~\forall ~i \in I_{0-}(x^\ast)~~ {\rm and}~~\eta_i^H ~ {\rm is}~ { \rm free}~ \forall~ i \in I_{0+}(x^\ast),\\
\eta_i^H &\geqslant& 0,~~ \eta_i^G = 0~~~~\forall i \in I_{00}(x^\ast),
\end{eqnarray*}
	${\rm (iii)} $ If $ \lambda, \mu, \eta^G, \eta^H $ not all equal to zero, then there is a sequence $\{x^k\} \rightarrow x^\ast$ such that  $\forall ~k \in \mathbb{N}$, we have
	\begin{eqnarray*}
				 \lambda_i &>&0  ~~\Rightarrow ~~\lambda_i g_i(x^k) >  0 ~~~ \{i=1,...,m\},\\
		 \mu_i &\neq&0  ~~\Rightarrow ~~\mu_i h_i(x^k) >  0 ~~~ \{i=1,...,p\},\\
		 \eta^H_i &\neq&0  ~~\Rightarrow ~~\eta^H_i H_i(x^k) <  0 ~~~ \{i=1,...,q\},\\
		 \eta^G_i &>&0  ~~\Rightarrow ~~\eta^G_i G_i(x^k) >  0 ~~~ \{i=1,...,q\}.	
	\end{eqnarray*}	
\end{Definition}
These enhanced conditions are stronger than those classic conditions. It is interesting to note that enhanced M-stationarity, being stronger than M-stationarity, is still weaker than S-stationarity (equivalently standard KKT, see \cite{achtziger}), in the sense that $\eta_i^G   \eta_i^H = 0 ~\forall~ i \in I_{00}(x^\ast)$ and other conditions in enhanced M-stationarity need not imply $\eta_i^H \geqslant 0,\eta_i^G = 0 ~\forall~ i \in I_{00}(x^\ast)$ to be S- stationarity.
\begin{Remark}
	From the above two definitions it is obvious that at any feasible point of MPVC, \\
	    \hspace*{3.5cm} enhanced S-stationarity $\Rightarrow$ enhanced M-stationarity.
\end{Remark}

\hspace*{0.5 cm}	Now, we  are in a position to consider some more constraint qualifications similar to pseudonormality and quasinormality concepts introduced by Bertsekas and Ozdaglar \cite{bertsekas2}, which   essentially  explore the behaviour of the problem in a neighbourhood  of solution points. These constraint qualifications  have been extended to the nonsmooth case by Ye and Zhang \cite{ye 1}, in terms of limiting subdifferential and also for MPEC as generalized pseudonormality and generalized quasinormality in \cite{ye}, by following the smooth version of MPEC notions, which were first introduced in \cite{Kanzow1}. These constraint qualifications associated with MPVC essentially play the same role as in MPEC case. Moreover, one of them will serve as a new constraint qualification and provides a sufficient condition for local error bound for MPVC (see Theorem \ref{thm last}). The following constraint qualification was introduced in \cite{hu} and shown to be a sufficient condition for exactness of the classical $l_1$ penalty function for MPVC under a reasonable assumption.
	\begin{Definition}
	\label{Def4} A vector $x^\ast \in \mathcal{C}$ is said to satisfy \textit{MPVC-generalized pseudonormality}, if there is no multiplier $(\lambda, \mu, \eta^H, \eta^G)~\neq~0$ such that \\
	
	$  {\rm (i)} ~\sum_{i=1}^{m} \lambda_i \nabla g_i(x^\ast) + \sum_{i=1}^{p} \mu_i \nabla h_i (x^\ast) + \sum_{i=1}^{q} \eta^G_i \nabla G_i(x^\ast)
	-\sum_{i=1}^{q} \eta^H_i \nabla H_i(x^\ast)~=~0$,\\
	
	$ {\rm (ii)} ~~~\lambda_i \geqslant 0 ~~~\forall~i \in I_g(x^\ast),~~~\lambda_i = 0 ~~~\forall~i \notin I_g(x^\ast)$
\begin{eqnarray*}
{and}~~~\eta_i^G &=& 0 ~~~~\forall~i \in I_{+-}(x^\ast) \cup I_{0-}(x^\ast) \cup I_{0+}(x^\ast),~~\eta_i^G  \geqslant 0~\forall i \in I_{+0}(x^\ast) \cup I_{00}(x^\ast),\\
\eta_i^H &=& 0 ~~~~\forall~i \in I_+(x^\ast),~~\eta_i^H  \geqslant  0~\forall ~i \in I_{0-}(x^\ast)~~ {\rm and}~~\eta_i^H ~ {\rm is}~ { \rm free}~ \forall~ i \in I_{0+}(x^\ast),\\
\eta_i^H  \eta_i^G &=& 0~~~~\forall i \in I_{00}(x^\ast),
\end{eqnarray*}
	
	$ {\rm (iii)} $ there is a sequence $\{x^k\} \rightarrow x^\ast$ such that the following is true for all $k \in \mathbb{N}$ \\
	\begin{equation*}
	\sum_{i=1}^{m} \lambda_i g_i(x^k)~+~\sum_{i=1}^{p} \mu_i h_i(x^k)~+~\sum_{i=1}^{q} \eta_i^G G_i(x^k)~-~\sum_{i=1}^{q} \eta_i^H H_i(x^k)~>~0.
	\end{equation*}
	\end{Definition}
	Now, we introduce a new constraint qualification, called MPVC-generalized  quasinormality analogous to MPEC-generalized quasinormality.
	\begin{Definition}
		 A vector $x^\ast \in \mathcal{C}$ is said to satisfy \textit{MPVC-generalized quasinormality}, if there is no multiplier $(\lambda, \mu, \eta^H, \eta^G)~\neq~0$ such that \\
		
		$ {\rm (i)} \sum_{i=1}^{m} \lambda_i \nabla g_i(x^\ast) + \sum_{i=1}^{p} \mu_i \nabla h_i (x^\ast) + \sum_{i=1}^{q} \eta^G_i \nabla G_i(x^\ast)
		-\sum_{i=1}^{q} \eta^H_i \nabla H_i(x^\ast)~=~0$,\\
		
		$ {\rm (ii)} ~~~\lambda_i \geqslant 0 ~~~\forall~i \in I_g(x^\ast),~~~\lambda_i = 0 ~~~\forall~i \notin I_g(x^\ast)$
\begin{eqnarray*}
{and}~~~\eta_i^G &=& 0 ~~~~\forall~i \in I_{+-}(x^\ast) \cup I_{0-}(x^\ast) \cup I_{0+}(x^\ast),~~\eta_i^G  \geqslant 0~\forall i \in I_{+0}(x^\ast) \cup I_{00}(x^\ast),\\
\eta_i^H &=& 0 ~~~~\forall~i \in I_+(x^\ast),~~\eta_i^H  \geqslant  0~\forall ~i \in I_{0-}(x^\ast)~~{\rm and}~~\eta_i^H ~ {\rm is}~ { \rm free}~ \forall~ i \in I_{0+}(x^\ast),\\
\eta_i^H  \eta_i^G &=& 0~~~~\forall i \in I_{00}(x^\ast),
\end{eqnarray*}					
$ {\rm (iii)} $ There is a sequence $\{x^k\} \rightarrow x^\ast$ such that the following is true $\forall k \in \mathbb{N} $, we have
\begin{eqnarray*}
		 \lambda_i &>&0  ~~\Rightarrow ~~\lambda_i g_i(x^k) >  0 ~~~ \{i=1,...,m\},\\
		 \mu_i &\neq&0  ~~\Rightarrow ~~\mu_i h_i(x^k) >  0 ~~~ \{i=1,...,p\},\\
		 \eta^H_i &\neq&0  ~~\Rightarrow ~~\eta^H_i H_i(x^k) <  0 ~~~ \{i=1,...,q\},\\
		 \eta^G_i &>&0  ~~\Rightarrow ~~\eta^G_i G_i(x^k) >  0 ~~~ \{i=1,...,q\}.	
	\end{eqnarray*}	
\end{Definition}
	We show, first time, that the MPVC-generalized  quasinormality is a sufficient condition for the existence of a local error bound of the MPVC, see Theorem \ref{lebthm}.\\
\hspace*{0.5 cm} To this end, we note that the MPVC-generalized pseudonormality obviously implies the MPVC-generalized  quasinormality, but not conversely. Combining this fact with \cite[Proposition 2.1]{hu}, we have following relationships among these constraint qualifications:\\\\
	MPVC-LICQ $\Rightarrow$ MPVC-MFCQ $\Rightarrow$	MPVC-GMFCQ $\Rightarrow$ MPVC-generalized pseudonormality $\Rightarrow$ MPVC-generalized quasinormality.\\
	
 \begin{Theorem}
 		Let $x^\ast$ be a local minimum of MPVC  satisfying MPVC-generalized quasinormality. Then $x^\ast$ is an enhanced M-stationary point.
 	\label{m stat}
 \end{Theorem}
 \begin{proof}
 	Suppose that $x^\ast$ is a local minimum of MPVC, then by enhanced Fritz John type  M-stationary conditions of Theorem \ref{thmefj-M}, we have
 	\begin{equation*}
 	\alpha \nabla f(x^\ast) + \sum_{i=1}^{m} \lambda_i \nabla g_i(x^\ast) + \sum_{i=1}^{p} \mu_i \nabla h_i (x^\ast) - \sum_{i=1}^{q} \eta_i^H \nabla H_i(x^\ast) + \sum_{i=1}^{q} \eta_i^G \nabla G_i(x^\ast)~=~0
 	\end{equation*}
 	and remaining results of the optimality conditions  also hold including $(\alpha, \lambda, \mu, \eta^G, \eta^H)~\neq~(0,0,0,0,0)$.\\
 	\hspace*{0.5cm} Here if $\alpha = 0$, then it shows the existence of nonzero multipliers which violates the MPVC-generalized quasinormality condition. Hence $\alpha > 0$, and then by proper scaling we obtain the enhanced M-stationary conditions at $x^\ast$.
 \end{proof}
 \begin{Corollary}
If $x^\ast$ is a local minimizer of MPVC satisfying MPVC-GMFCQ or MPVC-generalized pseudonormality, then $x^\ast$ is an enhanced M-stationary point.
 \end{Corollary}

 \begin{Theorem}
 	Suppose that $h_i$ are linear, $g_j$ are concave, $G_l, ~H_l$ are all linear. Then any feasible point of MPVC is MPVC-generalized pseudonormal.
 \end{Theorem}
 \begin{proof}
 We prove this result by contradiction. Suppose there is a feasible point $x^\ast$ that is not MPVC-generalized pseudonormal. Then there is a nonzero multiplier $(\lambda, \mu, \eta^G, \eta^H)$ such that
 \begin{equation}
 ~\sum_{j=1}^{m} \lambda_j \nabla g_j(x^\ast) + \sum_{i=1}^{p} \mu_i \nabla h_i (x^\ast) + \sum_{l=1}^{q} \eta^G_l \nabla G_l(x^\ast)
 -\sum_{l=1}^{q} \eta^H_l \nabla H_l(x^\ast)~=~0
 \label{lin 1}
 \end{equation}
 \begin{eqnarray*}
 {\rm and~~~}\lambda_j &\geqslant&0 ~~~~\forall~j \in I_g(x^\ast),~~~\lambda_j =0 ~~~~\forall~j \notin I_g(x^\ast),\\
 	\eta_l^G &=&0 ~~~~\forall~l \in I_{+-}(x^\ast) \cup I_{0-}(x^\ast) \cup I_{0+}(x^\ast),~~~\eta_l^G  \geqslant 0 ~~~~\forall~l \in I_{00}(x^\ast) \cup I_{+0}(x^\ast),\\	
 \eta_l^H &=& 0 ~~~~\forall~l \in I_+(x^\ast),~~~ \eta_l^H  \geqslant 0 ~~~~\forall~l \in I_{0-}(x^\ast),~{\rm and}~~~\eta_l^H ~{\rm is} ~{\rm free}~\forall~l \in I_{0+}(x^\ast),\\
 \eta_l^G \eta_l^H & = &0  ~~~~\forall~l \in I_{00}(x^\ast),
 \end{eqnarray*}
 and there is a sequence $\{x^k\} \rightarrow x^\ast$ such that the following is true for all $k \in \mathbb{N}$
 \begin{equation}
 \sum_{j=1}^{m} \lambda_j g_j(x^k)~+~\sum_{i=1}^{p} \mu_i h_i(x^k)~+~\sum_{l=1}^{q} \eta_l^G G_l(x^k)~-~\sum_{l=1}^{q} \eta_l^H H_l(x^k)~>~0.
 \label{lin 2}
 \end{equation}
 Now by the linearity of $h_i, G_l, H_l$ and by the concavity of $g_j$, we have  for all $x \in \mathbb{R}^n$
 \begin{eqnarray*}
 	h_i(x)&=& h_i(x^\ast) + \nabla h_i(x^\ast)^T (x- x^\ast)~~ \forall~  i = 1,...,p,\\
 	G_l(x)&=& G_l(x^\ast) + \nabla G_l(x^\ast)^T (x- x^\ast)~~ \forall~  l = 1,...,q,\\
 	H_l(x)&=& H_l(x^\ast) + \nabla H_l(x^\ast)^T (x- x^\ast)~~\forall ~ l = 1,...,q,\\
 	g_j(x)&\leqslant& g_j(x^\ast) + \nabla g_j^T (x^\ast) (x- x^\ast)~~\forall ~ j=1,...,m. 	
 \end{eqnarray*}
 By multiplying these relations with $\mu_i, \eta_l^G, \eta_l^H$ and $\lambda_j$ and adding over $i, l$ and $j$ respectively, we obtain $\forall ~ x \in \mathbb{R}^n$,
 \begin{eqnarray*}
 &&\sum_{j=1}^{m} \lambda_j \nabla g_j(x) + \sum_{i=1}^{p} \mu_i \nabla h_i (x) + \sum_{l=1}^{q} \eta^G_l \nabla G_l(x)
 -\sum_{l=1}^{q} \eta^H_l \nabla H_l(x)\\
 &\leqslant&~\sum_{j=1}^{m} \lambda_j \nabla g_j(x^\ast) + \sum_{i=1}^{p} \mu_i \nabla h_i (x^\ast) + \sum_{l=1}^{q} \eta^G_l \nabla G_l(x^\ast)
 -\sum_{l=1}^{q} \eta^H_l \nabla H_l(x^\ast) \\
 &&+\left [\sum_{j=1}^{m} \lambda_j \nabla g_j^T(x^\ast) + \sum_{i=1}^{p} \mu_i \nabla h_i(x^\ast)^T + \sum_{l=1}^{q} \eta_l^G \nabla G_l(x^\ast)^T - \sum_{l=1}^{q} \eta_l^H \nabla H_l(x^\ast)^T \right ](x-x^\ast)\\
 &=& \left [\sum_{j=1}^{m} \lambda_j \nabla g_j(x^\ast) + \sum_{i=1}^{p} \mu_i \nabla h_i(x^\ast) + \sum_{l=1}^{q} \eta_l^G \nabla G_l(x^\ast) - \sum_{l=1}^{q} \eta_l^H \nabla H_l(x^\ast) \right ]^T(x-x^\ast)
 \end{eqnarray*}
 the last inequality holds, because we have
 \begin{eqnarray*}
 \mu_i h_i(x^\ast) = 0 ~~~\forall~i &{\rm and}& ~~\sum_{j=1}^{m} \lambda_j g_j(x^\ast)=0,\\
 \sum_{l=1}^{q} \eta_l^G G_l(x^\ast) = 0, && \sum_{l=1}^{q} \eta_l^H H_l(x^\ast) = 0.	
 \end{eqnarray*}
 Now by the condition (\ref{lin 1}), we have
 \begin{equation*}
 	\sum_{j=1}^{m} \lambda_j \nabla g_j(x) + \sum_{i=1}^{p} \mu_i \nabla h_i (x) + \sum_{l=1}^{q} \eta^G_l \nabla G_l(x)
 	-\sum_{l=1}^{q} \eta^H_l \nabla H_l(x)~\leqslant~0 ~~\forall~x \in \mathbb{R}^n.
 \end{equation*}
 But, it contradicts condition (\ref{lin 2}), hence $x^\ast$ is MPVC-generalized pseudonormal.
 \end{proof}
 In \cite{qi},  Qi and Wei introduced $\mathit{Constant ~Positive~ Linear~ Dependence}$ (or CPLD) for standard nonlinear programs. Hoheisel et al. \cite{hohei con} introduced MPVC-CPLD for MPVC, which is weaker than MPVC-MFCQ. Further, it is  also generalized by Hoheisel et al.  in \cite{hoheisel 1} for MPEC, and later it has been employed to analyze these problems \cite{Kanzow 2,ye}.
  \begin{Definition}
  	\textbf{MPVC-CPLD:} A feasible point  $x^\ast$ is said to satisfy MPVC-CPLD if and only if for any indices set $I_h \subset \mathcal{B} = \{1,...,p\},~~J_0 \subset I_g(x^\ast), ~~L_0^H \subset I_{0+}(x^\ast) \cup I_{00}(x^\ast) \cup I_{0-}(x^\ast),~~L_0^G \subset I_{+0}(x^\ast) \cup I_{00}(x^\ast) $, whenever there exist $\lambda_j  ~\geqslant ~0 ~\forall ~j \in J_0, \mu_i, \eta_l^H$ and $\eta^G_l$ not all zero such that
  	\begin{equation*}
  	\sum_{j \in J_0}^{} \lambda_j \nabla g_j(x^\ast) + \sum_{i \in I_h}^{} \mu_i \nabla h_i (x^\ast) + \sum_{l \in L_0^G}^{} \eta^G_l \nabla G_l(x^\ast)
  	-\sum_{l \in L_0^H}^{} \eta^H_l \nabla H_l(x^\ast) ~=~0
  	\end{equation*}
  	and $\eta_l^G \eta_l^H = 0,~\forall~l \in I_{00}(x^\ast)$, then there is a neighbourhood $U(x^\ast)$ of $x^\ast$ such that for any $x \in U(x^\ast)$, the vectors
  	\begin{equation*}
  	\{\nabla h_i(x) | i \in I_h \}, ~\{\nabla g_j(x) | j \in J_0 \}, ~\{\nabla G_l(x) | l \in L_0^G \},~ \{\nabla H_l(x) | l \in L_0^H \}
  	\end{equation*}
  	are linearly dependent.
  \end{Definition}
  The following Lemma from \cite[Lemma 1]{andreani} is crucial to prove our next result.
  \begin{Lemma}
  	
  	If $x = \sum_{i=1}^{m+p} \alpha_i v_i$ with $v_i \in \mathbb{R}^n$ for every $i$, $\{v_i\}_{i=1}^m$ is linearly independent and $\alpha_i \neq 0$ for every $i = m+1,...,m+p$, then there exist $J \subset \{m+1,..., m+p\}$ and scalars $\bar{\alpha}_i$ for every $i \in \{1,...,m\} \cup J$ such that\\
  	${\rm (i)} ~~~ x = \sum_{i \in \{1,...,m\} \cup J}^{} \bar{\alpha}_i v_i$;\\
  	${\rm (ii)} ~~~ \alpha_i \bar{\alpha}_i > 0$ for every $i \in J$; \\
  	${\rm (iii)} ~~~ \{v_i\}_{i \in \{1,...,m\} \cup J}$ is linearly independent.
  	\label{lemma cpld}
  \end{Lemma}
  Now, using the definition of MPVC-CPLD and above Lemma, we have the following important result.
 \begin{Theorem}
 	Let $x$ be a feasible solution of MPVC such that MPVC-CPLD holds. Then $x$ is an MPVC-generalized quasinormal point.
 \end{Theorem}
 \begin{proof}
 Here, we deal only with vanishing constraints. Assume that $x$ is a feasible point and the MPVC-CPLD condition holds at $x$.\\
 \hspace*{0.5 cm} If $x$ satisfies MPVC-GMFCQ, then MPVC-generalized quasinormality obviously holds.\\
 Suppose, MPVC-GMFCQ does not hold, then there is a nonzero vector $(\eta^G, \eta^H) \in \mathbb{R}^q \times \mathbb{R}^q$ such that
 \begin{equation*}
 	- \sum_{l=1}^{q} \eta_l^H \nabla H_l(x) + \sum_{l=1}^{q} \eta_l^G \nabla G_l(x) ~=~0
 \end{equation*}
 and
 $ \eta^H_l = 0~\forall~ l \in I_+ (x), ~~\eta^G_l = 0~\forall~l \in I_{+-}(x) \cup I_{0-}(x) \cup I_{0+}(x)$, $\eta^H_l \geqslant 0~\forall~ l \in I_{0-}(x)$, and $\eta^H_l$ is free $\forall ~l \in I_{0+}(x) \cup I_{00}(x)$, $\eta^G_l \geqslant 0~\forall ~l \in I_{+0}(x) \cup I_{00}(x)$ and  $ \eta^G_l \eta^H_l = 0 ~\forall~l \in I_{00}(x).$\\\\
 Now we define the index sets
 \begin{eqnarray*}
 	A_+^H (x) &:=&\{ l \in I_{0-}(x) \cup I_{0+}(x) \cup I_{00}(x) | \eta_l^H > 0 \},\\
 	A_-^H (x) &:=&\{ l \in I_{0+}(x) \cup I_{00}(x) | \eta_l^H < 0 \},\\
 	A_+^G (x) &:=&\{ l \in I_{+0}(x)  \cup I_{00}(x) | \eta_l^G > 0 \},\\
 	I_{00}^{+0} (x) &:=&\{ l \in I_{00}(x) | \eta_l^H > 0, \eta^G_l = 0 \},\\
 	I_{00}^{-0} (x) &:=&\{ l \in I_{00}(x) | \eta_l^H < 0, \eta^G_l = 0 \},\\
 	I_{00}^{0+} (x) &:=&\{ l \in I_{00}(x) | \eta_l^H = 0, \eta^G_l > 0 \},\\
 	I_{00}^{0-} (x) &:=&\{ l \in I_{00}(x) | \eta_l^H = 0, \eta^G_l < 0 \}	.	
 \end{eqnarray*}
 Since $(\eta^H, \eta^G)$ is nonzero vector, therefore the union of the above sets must be nonempty and we may write
 \begin{eqnarray*}
 	0 &=&  -\left[ \sum_{l \in A^H_+(x)}^{} \eta_l^H \nabla H_l(x) + \sum_{l \in A_H^-(x)}^{} \eta_l^H \nabla H_l(x) \right] + \sum_{l \in A_+^G(x)}^{} \eta_l^G \nabla G_l(x)\\
 	&& - \left[ \sum_{l \in I_{00}^{+0}(x)}^{} \eta_l^H \nabla H_l(x) + \sum_{l \in I_{00}^{-0}(x)}^{} \eta_l^H \nabla H_l(x) \right] + \left[ \sum_{l \in I_{00}^{0+}(x)}^{} \eta_l^G \nabla G_l(x) + \sum_{l \in I_{00}^{0-}(x)}^{} \eta_l^G \nabla G_l(x) \right].
 \end{eqnarray*}
 First we assume that $A^H_+(x)$ is nonempty. Let $l_1 \in A^H_+(x)$, then
 \begin{eqnarray*}
 -\eta_{l_1}^H \nabla H_{l_1}(x) &=&  \left[ \sum_{l \in A^H_+(x)/ \{l_1\}}^{} \eta_l^H \nabla H_l(x) + \sum_{l \in A_H^-(x)}^{} \eta_l^H \nabla H_l(x) \right] - \sum_{l \in A_+^G(x)}^{} \eta_l^G \nabla G_l(x) \\
 &&+ \left[ \sum_{l \in I_{00}^{+0}(x)}^{} \eta_l^H \nabla H_l(x) + \sum_{l \in I_{00}^{-0}(x)}^{} \eta_l^H \nabla H_l(x) \right] - \left[ \sum_{l \in I_{00}^{0+}(x)}^{} \eta_l^G \nabla G_l(x) + \sum_{l \in I_{00}^{0-}(x)}^{} \eta_l^G \nabla G_l(x) \right].
 \end{eqnarray*}
 If $\nabla G_{l_1}(x) = 0$, then $\{\nabla G_{l_1}(x)\}$ is linearly dependent. Then, by MPVC-CPLD, set $\{\nabla G_{l_1}(y)\}$ must be linearly dependent for all $y$ in some neighbourhood of $x$. Therefore $\nabla G_{l_1}(y) = 0$ for all $y$ in an open neighbourhood of $x$. Since $G_{l_1}(x) = 0 ~ \Rightarrow ~ G_{l_1}(y) = 0$ for all $y$ in the neighbourhood of $x$. Hence, for any sequence $x^k \rightarrow x , ~G_{l_1}(x^k) = 0 $ for all sufficiently large $k$ always holds, i.e for sequence $x^k \rightarrow x$, $- \eta_l^H H_l(x^k) > 0$ never holds. Therefore, MPVC-generalized quasinormality holds at $x$.\\
 \hspace*{0.5 cm} Now, if $\nabla G_{l_1} \neq 0$, then $\nabla G_{l_1}(x)$ is linearly independent and then by lemma \ref{lemma cpld} there exists index sets
 \begin{eqnarray*}
 &&\tilde{A}^H_+(x) \subset A^H_+(x) / \{l_1\}, ~~~ \tilde{A}^H_-(x) \subset A^H_-(x) , ~~~ \tilde{A}^G_+(x) \subset A^G_+(x), ~~~ \tilde{I}_{00}^{+0}(x) \subset I_{00}^{+0}(x), \\
 && \tilde{I}_{00}^{-0}(x) \subset I_{00}^{-0}(x), ~~~
 \tilde{I}_{00}^{0+}(x) \subset I_{00}^{0+}(x), ~~~\tilde{I}_{00}^{0-}(x) \subset I_{00}^{0-}(x)
 \end{eqnarray*}
 such that the vectors
 \begin{eqnarray*}
 &&\{ \nabla H_l(x)\}_{l \in \tilde{A}_+^H(x)}, ~~~\{ \nabla H_l(x)\}_{l \in \tilde{A}_-^H(x)}, ~~~\{ \nabla G_l(x)\}_{l \in \tilde{A}_+^G(x)}, ~~~ \{ \nabla H_l(x)\}_{l \in \tilde{I}_{00}^{+0}(x)},\\
 &&\{ \nabla H_l(x)\}_{l \in \tilde{I}_{00}^{-0}(x)},~~~\{ \nabla G_l(x)\}_{l \in \tilde{I}_{00}^{0+}(x)},~~~\{ \nabla G_l(x)\}_{l \in \tilde{I}_{00}^{0-}(x)}
 \end{eqnarray*}
  are linearly dependent and
    \begin{eqnarray*}
  	-\eta_{l_1}^H \nabla H_{l_1}(x) &=&  \left[ \sum_{l \in \tilde{A}^H_+(x) }^{} \tilde{\eta}_l^H \nabla H_l(x) + \sum_{l \in \tilde{A}_H^-(x)}^{} \tilde{\eta}_l^H \nabla H_l(x) \right] - \sum_{l \in \tilde{A}_+^G(x)}^{} \tilde{\eta}_l^G \nabla G_l(x) \\
  	&&+ \left[ \sum_{l \in \tilde{I}_{00}^{+0}(x)}^{} \tilde{\eta}_l^H \nabla H_l(x) + \sum_{l \in \tilde{I}_{00}^{-0}(x)}^{} \tilde{\eta}_l^H \nabla H_l(x) \right] - \left[ \sum_{l \in \tilde{I}_{00}^{0+}(x)}^{} \tilde{\eta}_l^G \nabla G_l(x) + \sum_{l \in \tilde{I}_{00}^{0-}(x)}^{} \tilde{\eta}_l^G \nabla G_l(x) \right]
  \end{eqnarray*}

 with $\tilde{\eta}_l^H > 0  ~\forall ~l \in \tilde{A}_+^H(x),~~ \tilde{\eta}_l^H < 0  ~\forall ~l \in \tilde{A}_-^H(x),~~ \tilde{\eta}_l^G > 0  ~\forall ~l \in \tilde{A}_+^G(x),~~ \tilde{\eta}_l^H > 0,  ~\tilde{\eta}_l^G = 0 ~\forall ~l \in \tilde{I}_{00}^{+0}(x),~~ \tilde{\eta}_l^H < 0,  ~\tilde{\eta}_l^G = 0 ~\forall ~l \in \tilde{I}_{00}^{-0}(x), ~~\tilde{\eta}_l^H = 0,  ~\tilde{\eta}_l^G > 0 ~\forall ~l \in \tilde{I}_{00}^{0+}(x), ~~\tilde{\eta}_l^H = 0,  ~\tilde{\eta}_l^G < 0 ~\forall ~l \in \tilde{I}_{00}^{0-}(x) $.\\

 Now by the linear independence of the vectors and by the continuity argument, the vectors \\

  $\{ \nabla H_l(y)\}_{l \in \tilde{A}_+^H(x)}, ~~~\{ \nabla H_l(y)\}_{l \in \tilde{A}_-^H(x)}, ~~~\{ \nabla G_l(y)\}_{l \in \tilde{A}_+^G(x)}, ~~~ \{ \nabla H_l(y)\}_{l \in \tilde{I}_{00}^{+0}(x)}, \\
  \{ \nabla H_l(y)\}_{l \in \tilde{I}_{00}^{-0}(x)},~~~\{ \nabla G_l(y)\}_{l \in \tilde{I}_{00}^{0+}(x)},~~~\{ \nabla G_l(y)\}_{l \in \tilde{I}_{00}^{0-}(x)}$\\

  are linearly independent for all $y$ in a neighbourhood of $x$ and by the MPVC-CPLD assumption the vectors \\\\
  $ \eta_{l_1}^H \nabla H_{l_1}(y), ~~~ \{ \nabla H_l(y)\}_{l \in \tilde{A}_+^H(x)}, ~~~\{ \nabla H_l(y)\}_{l \in \tilde{A}_-^H(x)}, ~~~\{ \nabla G_l(y)\}_{l \in \tilde{A}_+^G(x)}, ~~~ \{ \nabla H_l(y)\}_{l \in \tilde{I}_{00}^{+0}(x)}, \\
  \{ \nabla H_l(y)\}_{l \in \tilde{I}_{00}^{-0}(x)},~~~\{ \nabla G_l(y)\}_{l \in \tilde{I}_{00}^{0+}(x)},~~~\{ \nabla G_l(y)\}_{l \in \tilde{I}_{00}^{0-}(x)}$\\\\
  are linearly dependent for all $y$ in a neighbourhood of $x$. Hence, $\eta_{l_1}^H \nabla H_{l_1}(y)$ must be a linear combination of all the remaining vectors for all $y$ in the neighbourhood of $x$.\\
  Now by \cite[Lemma 3.2]{andreani 2}, there exist a smooth function $\psi$ defined in a neighbourhood of $(0,...0)$ such that, for all $y$ in the neighbourhood of $x$,
  \begin{eqnarray*}
  -\eta_{l_1}^{H} \nabla H_{l_1}(y) &=& \psi \Big(\{ \nabla H_l(y)\}_{l \in \tilde{A}_+^H(x)}, ~\{ \nabla H_l(y)\}_{l \in \tilde{A}_-^H(x)}, ~\{ \nabla G_l(y)\}_{l \in \tilde{A}_+^G(x)}, ~ \{ \nabla H_l(y)\}_{l \in \tilde{I}_{00}^{+0}(x)}, \\
  && ~~\{ \nabla H_l(y)\}_{l \in \tilde{I}_{00}^{-0}(x)},~\{ \nabla G_l(y)\}_{l \in \tilde{I}_{00}^{0+}(x)},~\{ \nabla G_l(y)\}_{l \in \tilde{I}_{00}^{0-}(x)}\Big)\\
   \end{eqnarray*}
  and
  \begin{eqnarray*}
  \nabla \psi (0,...,0) &=& \Big( \{ \tilde{\eta}^H_l\}_{l \in \tilde{A}_+^H(x)}, ~\{ \tilde{\eta}^H_l\}_{l \in \tilde{A}_-^H(x)}, ~\{ \tilde{\eta}^G_l\}_{l \in \tilde{A}_+^G(x)}, ~ \{ \tilde{\eta}^H_l\}_{l \in \tilde{I}_{00}^{+0}(x)}, \\
  && ~~\{ \tilde{\eta}^H_l\}_{l \in \tilde{I}_{00}^{-0}(x)},~\{ \tilde{\eta}^G_l\}_{l \in \tilde{I}_{00}^{0+}(x)},~\{ \tilde{\eta}^G_l\}_{l \in \tilde{I}_{00}^{0-}(x)}\Big).
  \end{eqnarray*}
  Now, suppose $\{x^k\}$ is an infeasible sequence that converges to $x$ and such that
  \begin{eqnarray*}
  	H_l(x^k) ~<~0~~ &\forall& l \in \tilde{A}^H_+(x),\\
  	H_l(x^k) ~>~0 ~~&\forall& l \in \tilde{A}^H_-(x),\\
  	G_l(x^k) ~<~0 ~~&\forall& l \in \tilde{A}^G_+(x),\\
  	H_l(x^k) ~<~0 ~~&\forall& l \in \tilde{I}_{00}^{+0}(x),\\
  	H_l(x^k) ~>~0 ~~&\forall& l \in \tilde{I}_{00}^{-0}(x),\\
  	G_l(x^k) ~<~0 ~~&\forall& l \in \tilde{I}_{00}^{0+}(x),\\
  	G_l(x^k) ~>~0 ~~&\forall& l \in \tilde{I}_{00}^{0-}(x).
  \end{eqnarray*}
  Now by Taylor's expansion of $\psi$ at $(0,...,0)$, we have for $-\eta_{l_1}^H \nabla H_{l_1}(x^k)$ by the above sequence \\
  $-\eta_{l_1}^H \nabla H_{l_1}(x^k) = \psi(0,...,0) + \Big \langle \big( \{ \nabla H_l(x^k)\}_{l \in \tilde{A}_+^H(x)},..., \{ \nabla G_l(x^k)\}_{l \in \tilde{I}_{00}^{0-}(x)} \big),\\
  ~~~~~~~~~~~~~~~~~~~~~~~~~~~~~~~~~~~~~~~~~~~~\big(\{\tilde{\eta}_l^H\}_{l \in \tilde{A}_+^H(x)} ,..., \{\tilde{\eta}_l^G\}_{l \in \tilde{I}_{00}^{0-}(x)}\big) \Big \rangle $\\

  $ ~~~~~~~~~~~~~~ = \psi(0,...,0) + \sum_{l \in \tilde{A}^H_+(x)}^{} H_l(x^k) \tilde{\eta}_l^H + \sum_{l \in \tilde{A}^H_-(x)}^{} H_l(x^k) \tilde{\eta}_l^H + \sum_{l \in \tilde{A}^G_+(x)}^{} G_l(x^k) \tilde{\eta}_l^G \\
  ~~~~~~~~~~~~~~~~~~~~~~+ \sum_{l \in \tilde{I}^{+0}_{00}(x)}^{} H_l(x^k) \tilde{\eta}_l^H  + \sum_{l \in \tilde{I}^{-0}_{00}(x)}^{} H_l(x^k) \tilde{\eta}_l^H + \sum_{l \in \tilde{I}^{0+}_{00}(x)}^{} G_l(x^k) \tilde{\eta}_l^G \\
 ~~~~~~~~~~~~~~~~~~~~~~ + \sum_{l \in \tilde{I}^{0-}_{00}(x)}^{} G_l(x^k) \tilde{\eta}_l^G $\\\\
then for all $k$ large enough, we must have $-\eta_{l_1}^H \nabla H_{l_1}(x^k) \leqslant 0$, therefore for the sequence $x^k \rightarrow x$, the inequality  $-\eta_{l_1}^H \nabla H_{l_1}(x^k) > 0 $ does not hold. The  proofs for the remaining cases are entirely similar to the above case.  Hence, MPVC-generalized quasinormality holds.
 \end{proof}
 Since we already have seen that enhanced M-stationarity is a consequence of MPVC-generalized quasinormality. So, we have the following.
 \begin{Corollary}
 	Let $x^\ast$ be a local minimizer of MPVC. If $x^\ast$ satisfies MPVC-CPLD, then $x^\ast$ is an enhanced M-stationary point.
 \end{Corollary}
 \section{Local Error Bound}
 \label{leb}
 Since local error bound property is also a constraint qualification; hence  much attention has been paid in this context to standard nonlinear programs \cite{Pang97}, MPEC \cite{Kanzow1}, OPVIC \cite{ye 2} etc. Probably, for MPVC, \cite[Proposition 3.4]{hoheisel} is the first result on the existence of local error bound
 where it has been proved, but in context of a more general problem, that calmness of some perturbed map is equivalent to the existence of local error bounds. In \cite{Dussault}, the  local error bound result is derived under the MPVC-\emph{constant rank in the subspace of components} (MPVC-CRSC).  To find the relationship between MPVC-generalized quasinormality and MPVC-CRSC in \cite{Dussault}, it needs a separate discussion; we do not discuss it in this paper. We prove that the MPVC-generalized quasinormality is sufficient condition for the existence of local error bound in Theorem \ref{lebthm}.\\\\
 \hspace*{0.5 cm} In order to prove error bound result given in Theorem \ref{lebthm}, we collect some more results, which are necessary for the proof. The first result ensures the quasinormality in a whole neighbourhood.

  \begin{Lemma}
 	 \label{lemma a}
 If a feasible point $x^\ast$ is MPVC-generalized quasinormal, then all feasible points in a neighbourhood of $x^\ast$ are MPVC-generalized quasinormal.
\end{Lemma}
 \begin{proof}
 Suppose contrary that there is a sequence $\{x^k\}$ such that $x^k \neq x$ for all $k$ and $x^k \rightarrow x$ and $x^k$ is not MPVC-generalized quasinormal for all $k$. Therefore there exist scalars $(\lambda^k, \mu^k, \eta^{G^k}, \eta^{H^k}) \neq (0,0,0,0)$ and a sequence $\{x^{k,t}\}$ such that \\
  ${\rm (i)} ~~\sum_{i=1}^{m} \lambda_i^k \nabla g_i(x^k) + \sum_{i=1}^{p} \mu_i^k \nabla h_i (x^k) + \sum_{i=1}^{q} \eta^{G^k}_i \nabla G_i(x^k)
  -\sum_{i=1}^{q} \eta^{H^k}_i \nabla H_i(x^k)~=~0, $\\
  ${\rm (ii)}~~~\lambda_i^k \geqslant0 ~\forall~i \in I_g(x^k),~~~
  	\lambda_i^k = 0 ~~\forall~i \notin I_g(x^k)$
  \begin{eqnarray*}
  {\rm and~~~~}\eta_i^{G^k} &=&0 ~~~~~~\forall~i \in I_{+-}(x^k) \cup I_{0-}(x^k) \cup I_{0+}(x^k),
  ~~~\eta_i^{G^k} \geqslant 0 ~~~~\forall~i \in I_{00}(x^k) \cup I_{+0}(x^k),\\
  	\eta_i^{H^k} &=& 0 ~~~~\forall~i \in I_+(x^k),~~~ \eta_i^{H^k}  \geqslant 0 ~~~~\forall~i \in I_{0-}(x^k),~~~
  	\eta_i^{H^k}~ {\rm is}~{\rm~free~}~\forall~i \in I_{0+}(x^k),\\
  	\eta_i^{G^k}\eta_i^{H^k} & = &0  ~~~~\forall~i \in I_{00}(x^k),
  \end{eqnarray*}
  $ {\rm (iii)} $ There is a sequence $\{x^{k,t}\} \rightarrow x^k$ such that the following is true $\forall ~k  $ as $t \rightarrow \infty$,
  \begin{eqnarray*}
  	\lambda_i^k g_i(x^{k,t})  &>0& \forall ~\lambda_i^k > 0, \\
  	\mu_i^k h_i(x^{k,t}) &>0&\forall~\mu_i^k \neq 0, \\
  	-\eta_i^{H^k} H_i(x^{k,t})&>0& \forall ~ \eta_i^{H^k} \neq 0, \\
  	\eta_i^{G^k} G_i(x^{k,t})&>0& \forall ~ \eta_i^{G^k} > 0.
  \end{eqnarray*}
 Let for each $k$,
 \begin{eqnarray*}
 	\tilde{\lambda}^k &=& \frac{\lambda^k}{||(\lambda^k, \mu^k, \eta^{G^k}, \eta^{H^k})||}, \\
 	\tilde{\mu}^k &=& \frac{\mu^k}{||(\lambda^k, \mu^k, \eta^{G^k}, \eta^{H^k})||},\\
 	\tilde{\eta}^{G^k} &=& \frac{\eta^{G^k}}{||(\lambda^k, \mu^k, \eta^{G^k}, \eta^{H^k})||},\\
 	\tilde{\eta}^{H^k} &=& \frac{\eta^{H^k}}{||(\lambda^k, \mu^k, \eta^{G^k}, \eta^{H^k})||}.
 \end{eqnarray*}
 then without loss of generality we assume that
 \begin{equation*}
 	(\tilde{\lambda}^k, \tilde{\mu}^k, \tilde{\eta}^{G^k}, \tilde{\eta}^{H^k}) ~\rightarrow ~(\lambda^\ast, \mu^\ast, \eta^{G^\ast}, \eta^{H^\ast})
 \end{equation*}
 dividing both sides of (i), (ii) and (iii) by $||(\lambda^k, \mu^k, \eta^{G^k}, \eta^{H^k})||$ and taking the limit, we have \\

 ${\rm (i)} ~~\sum_{i=1}^{m} \lambda_i^\ast \nabla g_i(x^\ast) + \sum_{i=1}^{p} \mu_i^\ast \nabla h_i (x^\ast) + \sum_{i=1}^{q} \eta^{G^\ast}_i \nabla G_i(x^\ast)
 -\sum_{i=1}^{q} \eta^{H^\ast}_i \nabla H_i(x^\ast)~=~0 $,\\

 $ {\rm (ii)}~~~\lambda_i^\ast \geqslant 0 ~~~~\forall~i \in I_g(x^\ast),~~~
 	\lambda_i^\ast =0~\forall~i \notin I_g(x^\ast)$
 \begin{eqnarray*}
 {\rm and}~~~~	\eta_i^{G^\ast} &=&0 ~~~~\forall~i \in I_{+-}(x^\ast) \cup I_{0-}(x^\ast) \cup I_{0+}(x^\ast),~~~
 	\eta_i^{G^\ast}  \geqslant 0 ~~~~\forall~i \in I_{00}(x^\ast) \cup I_{+0}(x^\ast),\\
 	\eta_i^{H^\ast} &=& 0 ~~~~\forall~i \in I_+(x^\ast),~~~
 	\eta_i^{H^\ast}  \geqslant 0 ~~~~\forall~i \in I_{0-}(x^\ast),~~~
 	\eta_i^{H^\ast} ~{\rm is}~{\rm~ free}~ \forall~i \in I_{0+}(x^\ast), \\
 	\eta_i^{G^\ast}\eta_i^{H^\ast} & = &0  ~~~~\forall~i \in I_{00}(x^\ast),
 \end{eqnarray*}

 $ {\rm (iii)} $ A sequence $\{\xi^k\} \rightarrow x^\ast$ such that the following holds for each $i$ as $k \rightarrow \infty$,
 \begin{eqnarray*}
 	\lambda_i^\ast g_i(\xi^k) &>0&  \forall ~\lambda_i^\ast > 0, \\
 	\mu_i^\ast h_i(\xi^k) &>0& \forall ~\mu_i^\ast \neq 0, \\
 	-\eta_i^{H^\ast} H_i(\xi^k) &>0& \forall ~ \eta_i^{H^\ast} \neq 0, \\
 	\eta_i^{G^\ast} G_i(\xi^k)&>0& \forall ~ \eta_i^{G^\ast} > 0.
 \end{eqnarray*}
 since $(\lambda^k, \mu^k, \eta^{H^k}, \eta^{G^k})$ is nonzero, therefore $(\lambda^\ast, \mu^\ast, \eta^{H^\ast}, \eta^{G^\ast}) \neq (0,0,0,0)$. Hence, it is a contradiction to the fact that $x^\ast$ is MPVC-generalized quasinormal. Hence, all feasible points in a neighbourhood of $x^\ast$ are MPVC-generalized quasinormal.
 \end{proof}
 The following result gives a representation for the limiting normals at a point to the constraint region in terms of quasinormal multipliers

 \begin{Theorem}
 		\label{th a}
 	If $\bar{x}$ is MPVC-generalized quasinormal for $\mathcal{C}$, then
 	\begin{equation*}
 	N_\mathcal{C} (\bar{x}) \subset \left \{ \sum_{i=1}^{m} \lambda_i \nabla g_i(\bar{x}) + \sum_{i=1}^{p} \mu_i \nabla h_i(\bar{x}) +\sum_{i=1}^{q} [ \eta_i^G \nabla G_i (\bar{x}) -  \eta_i^H \nabla H_i(\bar{x}) ]  \Big |~ (\lambda, \mu,\eta^G, \eta^H) \in M(\bar{x})  \right \}
 	\end{equation*}
 	where $M(\bar{x})$ denotes the set of MPVC-generalized quasinormal multipliers corresponding to the point $\bar{x}$.

 \end{Theorem}
 \begin{proof}
 Here, for sake of the simplicity, we omit the equality and inequality constraints and consider only vanishing constraints, which needs to be handled. Let $v$ be an element of set $N_\mathcal{C}(\bar{x})$. By definition of limiting normal cone, there are sequences $x^l \rightarrow \bar{x}$ and $v^l \rightarrow v$ with $v^l \in N_\mathcal{C}^F (x^l)$ and $x^l \in \mathcal{C}$.\\
 \hspace*{0.5 cm} Step I. By the lemma \ref{lemma a}, $x^l$ is generalized quasinormal for sufficient large $l$. Now by \cite[Theorem 6.11]{Rocka}, for each $l$, there exists a smooth function $\psi^l$ that has a strict global minimizer $x^l$ over $\mathcal{C}$  with $-\nabla \psi^l (x^l) = v^l$. Since $x^l$ is MPVC-generalized quasinormal point of $\mathcal{C}$, therefore by Theorem \ref{m stat} enhanced M-stationary condition holds for problem
 \begin{eqnarray*}
 	\min ~\psi^l(x)\\
  	{\rm s.t.}~~x \in \mathcal{C}.
 \end{eqnarray*}
 That is, there exists a vector $(\eta^{G^l}, \eta^{H^l})$ such that
 \begin{equation}
 v^l \in \sum_{i=1}^{q} \eta_i^{G^l} \nabla G_i(x^l) - \sum_{i=1}^{q} \eta_i^{H^l} \nabla H_i(x^l)
 \label{eq normal}
 \end{equation}
 with
 \begin{eqnarray*}
 \eta_i^{G^l} &=&0 ~~\forall~i \in I_{+-}(x^l) \cup I_{0-}(x^l) \cup I_{0+}(x^l),~~~
 \eta_i^{G^l}  \geqslant 0 ~~\forall~i \in I_{00}(x^l) \cup I_{+0}(x^l),\\
 \eta_i^{H^l} &=& 0 ~~\forall~i \in I_+(x^l),~~~ \eta_i^{H^l} \geqslant 0 ~~~~ \forall~i \in I_{0-}(x^l), {~\rm and}~
 \eta_i^{H^l} ~{\rm is} ~{\rm free} ~\forall~ i \in I_{0+}(x^l),\\
 \eta_i^{G^l} \eta_i^{H^l} & = &0  ~~\forall~i \in I_{00}(x^l).
 \end{eqnarray*}
Moreover, let $\mathcal{G}^l = \{i | \eta^{G^l}_i > 0\},~   \mathcal{H}^l = \{i | \eta^{H^l}_i \neq 0 \}$ then there is a sequence $\{x^{k,l}\} \rightarrow x^k$ as $k \rightarrow \infty $ such that  $\forall ~k \in \mathbb{N}$,
 \begin{eqnarray*}
  \eta^{G^l}_i G_i(x^{k,l}) &>&0 ~~\forall ~i \in \mathcal{G}^l,\\
  -\eta^{H^l}_i H_i(x^{k,l}) &>&0 ~~\forall ~i \in \mathcal{H}^l.
 \end{eqnarray*}
 \hspace*{0.5 cm} Step II. Now we will show that the sequence $\{(\eta^ {G^l}, \eta^{H^l})\}$ is bounded. To prove this, on contrary suppose that  $\{(\eta^ {G^l}, \eta^{H^l})\}$  is unbounded. Now for all $l$, denote
 \begin{eqnarray*}
 	\tilde{\eta}^{G^l} &=& \frac{\eta^{G^l}}{||(\eta^{G^l}, \eta^{H^l})||},\\
 	\tilde{\eta}^{H^l} &=& \frac{\eta^{H^l}}{||(\eta^{G^l}, \eta^{H^l})||}
 \end{eqnarray*}
 then without loss of generality, we can assume that
 \begin{equation*}
 	(\tilde{\eta}^{G^l}, \tilde{\eta}^{H^l}) \rightarrow (\eta^{G^\ast}, \eta^{H^\ast}).
 \end{equation*}
 Dividing both sides of eq (\ref{eq normal}) by $||(\eta^{G^l}, \eta^{H^l})||$ and then taking the limit, we obtain
 \begin{equation}
 0 \in \sum_{i=1}^{q} \eta_i^{G^\ast} \nabla G_i(\bar{x}) - \sum_{i=1}^{q} \eta_i^{H^\ast} \nabla H_i(\bar{x})
 \end{equation}
 \begin{eqnarray*}
 {~\rm and~~~~~~}\eta_i^{G^\ast} &=&0 ~~~~\forall~i \in I_{+-}(\bar{x}) \cup I_{0-}(\bar{x}) \cup I_{0+}(\bar{x}),~~~
 	\eta_i^{G^\ast}  \geqslant 0 ~~~~\forall~i \in I_{00}(\bar{x}) \cup I_{+0}(\bar{x}),\\
 	\eta_i^{H^\ast} &=& 0 ~~~~\forall~i \in I_+(\bar{x}),~~~
 	\eta_i^{H^\ast}  \geqslant 0 ~~~~ \forall~i \in I_{0-}(\bar{x}),~{\rm and}~
 	\eta_i^{H^\ast}~ {\rm is}~{\rm~free}~\forall~i \in I_{0+}(\bar{x}),\\
 	\eta_i^{G^\ast} \eta_i^{H^\ast} & = &0  ~~~~\forall~i \in I_{00}(\bar{x})
 \end{eqnarray*}
 and a sequence $\{\zeta^l \} \rightarrow \bar{x}$ as $l \rightarrow \infty$, and for each $l$,
 \begin{eqnarray*}
 	\eta^{G^\ast}_i G_i(\zeta^l) &>0& ~~\forall ~i \in \mathcal{G}^l,\\
 	-\eta^{H^\ast}_i H_i(\zeta^l) &>0& ~~\forall ~i \in \mathcal{H}^l.
 \end{eqnarray*}
 Now, here is a violation of the fact that $\bar{x}$ is MPVC-quasinormal, hence sequence $\{(\eta^ {G^l}, \eta^{H^l})\}$ must be bounded.\\
  \hspace*{0.5 cm}Step III.  Without loss of generality we can assume now $(\eta^{G^l}, \eta^{H^l}) \rightarrow (\eta^G, \eta^H)$ as $l \rightarrow \infty$,
  \begin{equation}
  v \in \sum_{i=1}^{q} \eta_i^G \nabla G_i(\bar{x}) - \sum_{i=1}^{q} \eta_i^H \nabla H_i(\bar{x})
  \end{equation}
   \begin{eqnarray*}
  	{~~\rm and~~~~}\eta_i^G &=&0 ~~~~\forall~i \in I_{+-}(\bar{x}) \cup I_{0-}(\bar{x}) \cup I_{0+}(\bar{x}),~~~
  	\eta_i^G  \geqslant 0 ~~~~\forall~i \in I_{00}(\bar{x}) \cup I_{+0}(\bar{x}),\\
  	\eta_i^H &=& 0 ~~~~\forall~i \in I_+(\bar{x})),~~~
  	\eta_i^H  \geqslant 0 ~~~~ \forall~i \in I_{0-}(\bar{x}),~ {\rm and}~
  	\eta_i^H ~{\rm is} ~ {\rm~free} ~\forall~i \in I_{0+}(\bar{x}),\\
  	\eta_i^G \eta_i^H & = &0  ~~~~\forall~i \in I_{00}(\bar{x})
  \end{eqnarray*}
  and we can find a subsequence $\{\zeta^l\}$ converges to $\bar{x}$ as $l \rightarrow \infty $ and for each $l$,
  \begin{eqnarray*}
  	\eta^G_i G_i(\zeta^l) &>0& ~~\forall ~i \in \mathcal{G}^l,\\
  	-\eta^H_i H_i(\zeta^l) &>0& ~~\forall ~i \in \mathcal{H}^l.
  \end{eqnarray*}
  Hence, the proof is complete.
 \end{proof}

 In \cite[Theorem 4.5]{Kanzow1}, it has been established that MPEC-generalized pseudonormality is a sufficient condition for the existence of a local error bound for smooth MPEC. Ye and Zhang \cite[Theorem 3.1]{ye} has improved it for nonsmooth MPEC under the assumptions that $g_j$ are to be  only subdifferentially regular around the concerned point. In the context of MPVC, we confine ourselves only to smooth case, see \cite{Minch}. However, one can show that the result also holds under the assumptions of \cite[Theorem 3.1]{ye}.
 \begin{Theorem}
 \label{lebthm}
 	Let $x^\ast \in \mathcal{C}$ the feasible region of MPVC. If $x^\ast$ is MPVC-generalized quasinormal, then there are $\delta, c > 0$ such that
 	\begin{equation}
 	dist_\mathcal{C}(x) \leqslant c \left( ||h(x)||_1 + ||g^+(x)||_1 + \sum_{i=1}^{q} dist_\Delta (G_l(x), H_l(x)) \right)
 	\end{equation}
 	\begin{equation*}
 	holds~~\forall~~	x \in \mathbb{B}(x^\ast, \delta / 2) ,
 	\end{equation*}
 	where $\Delta := \{ (a,b) \in \mathbb{R}^2 | b \geqslant 0, ab \leqslant 0 \}$, and $~dist_\Delta (x)$ is the distance in $l_1$-norm from $x$ to the set $\Delta$.
 	\label{thm last}
 \end{Theorem}
 \begin{proof}
 For the sake of simplicity, we omit the equality constraints. Here, we check only the case when $x^\ast$ is on the boundary because assertion may fail there for $x \notin \mathcal{C}$. For $x \in int \mathcal{C}$ assertion is always true.\\
 \hspace*{0.5 cm} we choose some sequences $\{y^k\}$ and $\{x^k\}$ such that $y^k \rightarrow x^\ast,~y^k \notin \mathcal{C},$ and $x^k = \prod _\mathcal{C}(y^k)$, the projection of $y^k$ onto the set $\mathcal{C}$. Since $||x^k - y^k|| \leqslant ||y^k - x^\ast||$, therefore $x^k \rightarrow x^\ast$. We may assume here that both the sequences $\{y^k\}$ and $\{x^k\}$ belong to $\mathbb{B}(x^\ast, \delta_0)$.\\
 \hspace*{0.5 cm} Since $y^k - x^k \in N_\mathcal{C}^\pi (x^k) \subset N_\mathcal{C}^F(x^k)$, we have $\eta^k = \frac{y^k - x^k}{||y^k -x^k||} \in N_\mathcal{C}^F(x^k)$. Since $x^\ast$ is  MPVC-generalized quasinormal, therefore $x^k$ is also MPVC-generalized quasinormal (by Lemma \ref{lemma a}) for all sufficiently large $k$ and then without loss of generality, we may assume that all $x^k$ are MPVC-generalized quasinormal. Then by Theorem \ref{th a} there exist a sequence of scalars $\{(\lambda^k, \eta^{G^k}, \eta^{H^k})\}$ such that
 \begin{equation}
 \label{eq aa}
 	 \eta^k \in  \sum_{i=1}^{m} \lambda_i^k  \nabla g_i(x^k)  + \sum_{i=1}^{q} \eta_i^{G^k} \nabla G_i(x^k) - \sum_{i=1}^{q} \eta_i^{H^k} \nabla H_i(x^k)
 	  \end{equation}
 and~~~~$\lambda_i^k \geqslant 0 ~\forall~i \in I_g(x^k),~
 	~\lambda_i^k = 0 ~~ \forall~i \notin I_g(x^k),$
 \begin{eqnarray*}
 	\eta_i^{G^k} &=&0 ~~~~\forall~i \in I_{+-}(x^k) \cup I_{0-}(x^k) \cup I_{0+}(x^k),~~~
 	\eta_i^{G^k} \geqslant 0 ~~~~\forall~i \in I_{00}(x^k) \cup I_{+0}(x^k),\\
 	\eta_i^{H^k} &=& 0 ~~~~\forall~i \in I_+(x^k),~~~
 	\eta_i^{H^k}  \geqslant 0 ~~~~\forall~i \in I_{0-}(x^k),~{\rm and}~
 	\eta_i^{H^k} ~{\rm is}~{\rm free ~ } ~\forall~ i \in I_{0+}(x^k),\\
 	\eta_i^{G^k} \eta_i^{H^k} & = &0  ~~~~\forall~i~ \in I_{00}(x^k)
 \end{eqnarray*}
 and there exists a sequence $\{x^{k,s}\} \in \mathbb R^n$ such that $x^{k,s} \rightarrow x^k$ as $s \rightarrow \infty$ and for which we have
 \begin{eqnarray*}
 \lambda_i^k g_i(x^{k,s})&>&0 ~~~~~~~~~\forall ~\lambda_i^k > 0 ~~ {\rm and} ~~\forall ~s\in \mathbb N,\\
 \eta_i^{G^k} G_i(x^{k,s}) &>&0 ~~~~~~~~~ \forall ~\eta_i^{G^k} > 0~~ {\rm and} ~~\forall ~s\in \mathbb N, \\
 -\eta_i^{H^k} H_i(x^{k,s}) &>&0 ~~~~~~~~~\forall ~\eta_i^{H^k} \neq 0 ~~~ {\rm and} ~~\forall ~s\in \mathbb N.
 \end{eqnarray*}
 Now, similar to the proof of Theorem \ref{th a}, we can show that the MPVC-generalized quasinormality of $x^\ast$ shows the boundedness of sequence $\{(\lambda^k, \eta^{G^k}, \eta^{H^k})\}$, and hence we may assume that $(\lambda^k, \eta^{G^k}, \eta^{H^k})$ converges to some vector $( \lambda^\ast, \eta^{G^\ast}, \eta^{H^\ast} )$. Then there exists a number $M_0 > 0$ such that for all $k$, $||(\lambda^\ast, \eta^{G^\ast}, \eta^{H^\ast})|| \leqslant M_0$. Without loss of generality, we may assume that $y^k \in \mathbb{B}(x^\ast, \frac{\delta_0}{2})/\mathcal{C}$ and $x^k \in \mathbb{B}(x^\ast, \delta_0)$ for all $k$. Now we set $(\bar{\lambda}^k, \bar{\eta}^{G^k}, \bar{\eta}^{H^k}) = 2(\lambda^k, \eta^{G^k}, \eta^{H^k})$, then from (\ref{eq aa}), for each $k$,
 \begin{equation*}
 \eta^k = \frac{y^k-x^k}{||y^k - x^k||} = \sum_{i=1}^{m} \lambda^k \nabla g_i(x^k)	 + \sum_{i=1}^{q} \eta_i^{G^k} \nabla G_i(x^k) - \sum_{i=1}^{q} \eta_i^{H^k} \nabla H_i(x^k)
 \end{equation*}
 or
 \begin{equation*}
 \frac{y^k-x^k}{||y^k - x^k||} = \frac{x^k - y^k}{||y^k - x^k||}  + \sum_{i=1}^{m} \bar{\lambda}_i^k \nabla g_i(x^k)	 + \sum_{i=1}^{q} \bar{\eta}_i^{G^k} \nabla G_i(x^k) - \sum_{i=1}^{q} \bar{\eta}_i^{H^k} \nabla H_i(x^k).
 \end{equation*}
 Now from the above discussion, we have
 \begin{eqnarray*}
 	||y^k - x^k || &=&\left \langle \frac{y^k - x^k}{||y^k -x^k||}, y^k - x^k  \right \rangle \\
 	&=& \left \langle \frac{x^k - y^k}{||y^k -x^k||}, y^k - x^k  \right \rangle  + \sum_{i=1}^{m} \langle \bar{\lambda}_i^k \rho_i^k, y^k-x^k  \rangle \\
 	&&  + \sum_{i=1}^{q} \langle \bar{\eta}_i^{H^k} \nabla H_i (x^k), y^k - x^k \rangle - \sum_{i=1}^{q} \langle \bar{\eta}_i^{G^k} \nabla G_i (x^k), y^k - x^k \rangle \\
 	&\leqslant&\sum_{i=1}^{m} \langle \bar{\lambda}_i^k \nabla g_i(x^k), y^k-x^k  \rangle
 	  + \sum_{i=1}^{q} \langle \bar{\eta}_i^{H^k} \nabla H_i (x^k), y^k - x^k \rangle\\ &&- \sum_{i=1}^{q} \langle \bar{\eta}_i^{G^k} \nabla G_i (x^k), y^k - x^k \rangle + o (||y^k - x^k||) \\
 	  &\leqslant& \sum_{i=1}^{m}  \bar{\lambda}_i^k (g_i(y^k) + o(||y^k-x^k||))
 	  + \sum_{i=1}^{q}  \bar{\eta}_i^{H^k} ( H_i (y^k)+  o(||y^k-x^k||))\\
 	  &&- \sum_{i=1}^{q} \bar{\eta}_i^{G^k}  (G_i (y^k) + o (||y^k - x^k||))+ o (||y^k - x^k||)\\
 	  &\leqslant& 2 \left[ \sum_{i=1}^{m} \lambda_i^k g_i(y^k) + \sum_{i=1}^{q} \eta_i^{H^k} H_i(y^k) - \sum_{i=1}^{q} \eta_i^{G^k} G_i(y^k) \right]\\
 	  &&+ 2 \left |\sum_{i=1}^{m} \lambda_i^k  + \sum_{i=1}^{q} \eta_i^{H^k}  - \sum_{i=1}^{q} \eta_i^{G^k} + 1 \right | o (||y^k-x^k||).
 \end{eqnarray*}
 Now, without loss of generality, we may assume that for sufficiently large $k$,
 \begin{equation*}
 	o (||y^k-x^k||) ~\leqslant~\frac{1}{4(M_0 + 1)} ||y^k - x^k||,
 \end{equation*}
 then we have
 \begin{eqnarray*}
||y^k-x^k|| &\leqslant&2 \left[ \sum_{i=1}^{m} \lambda_i^k g_i(y^k) + \sum_{i=1}^{q} \eta_i^{H^k} H_i(y^k) - \sum_{i=1}^{q} \eta_i^{G^k} G_i(y^k) \right] + \frac{1}{2} ||y^k - x^k||.
 \end{eqnarray*}
 Now, all the above discussion implies that
 \begin{equation*}
 { \rm dist}_\mathcal{C}(y^k) ~=~||y^k - x^k|| ~ \leqslant ~ 4 M_0 \left( \sum_{i=1}^{m} g_i^+(y^k) + \phi (G(y^k), H(y^k)) \right)
 \end{equation*}
 where
 \begin{equation*}
 	\phi (G(y^k), H(y^k)) ~=~\sum_{i=1}^{q}\max \{0, -H_i(y^k), ~\min\{G_i(y^k),H_i(y^k)\} \}.
 \end{equation*}
 Hence, for any sequence $\{y^k\} \in \mathbb R^n$ converging to $x^\ast$ there is a number $c > 0$ such that
 \begin{equation*}
 	{ \rm dist}_\mathcal{C}(y^k) ~ \leqslant ~ c \left( ||g^+(y^k)||_1 + \sum_{i=1}^{q} {\rm dist}_\Delta (G_i(y^k), H_i(y^k)) \right)~~~~~\forall~k  \in \mathbb N.
 \end{equation*}
 This implies the error bound property at $x^\ast$. Indeed, suppose the contrary. Then there exists a sequence $\tilde{y}^k \rightarrow x^\ast$ such that $\tilde{y}^k \notin \mathcal{C}$ and
  \begin{equation*}
  { \rm dist}_\mathcal{C}(\tilde{y}^k) ~ > ~ c \left( ||g^+(\tilde{y}^k)||_1 + \sum_{i=1}^{q} {\rm dist}_\Delta (G_i(\tilde{y}^k), H_i(\tilde{y}^k)) \right)~~~~~\forall~k \in \mathbb N.
  \end{equation*}
  which is a contradiction.
   \end{proof}
  \section{Concluding Remarks}
  \label{concl}
  We have used the Fritz John  approach for MPVC, first time, to derive the M-stationary conditions under weak  constraint qualifications. The derivations for M-stationary conditions given in section 2 are simpler than others, available in the literature. Further, the enhanced M-stationarity  has been shown to be a new stationary condition for MPVC. The enhanced stationarity motivated to introduce a new  constraint qualification: MPVC-generalized quasinormality and is found to be weaker than MPVC-CPLD. An error bound result has been found using these constraint qualifications. However, it remains to discuss the relationship of this new constraint qualification with other known MPVC-constraint qualifications, such as MPVC-GCQ and MPVC-ACQ. We hope that these relationships will open up some new paths for MPVC field.

 	\end{document}